\documentclass{amsart}
\usepackage{amssymb,amsmath,amsfonts,amsthm}
\usepackage[dvips]{graphicx}
\usepackage{psfrag}
\usepackage[usenames, dvipsnames]{color}
\usepackage{tikz}

\theoremstyle{plain}
\newtheorem{main}{Theorem}

\newtheorem{theorem}{Theorem}[section]
\newtheorem{lemma}[theorem]{Lemma}
\newtheorem{proposition}[theorem]{Proposition}
\newtheorem{corollary}[theorem]{Corollary}
\theoremstyle{remark}
\newtheorem{remark}[theorem]{Remark}
\newtheorem{definition}[theorem]{Definition}

\newtheorem{conjecture}[theorem]{Conjecture}
\numberwithin{equation}{section}

\newcommand{\C}{\operatorname{C}}
\newcommand{\G}{\operatorname{G}}

\newcommand{\Gibb}{\operatorname{Gibb}}

           \def\ea{\end{array}}
          \def\ec{\end{center}}
     \def\ed{\end{description}}
        \def\ee{\end{equation}}
       \def\eea{\end{eqnarray}}
     \def\eeaa{\end{eqnarray*}}
 \def\et{\end{thebibliography}}
\def\bib{\bibitem}

\def\Diff{{\rm Diff}}
\def\Cl{{\rm Cl}}

\def\cG{{\mathcal G}}

\def\cC{{\mathcal C}}

\def\cU{{\mathcal U}}

\def\cB{{\mathcal B}}
\def\cD{{\mathcal D}}

\def\cF{{\mathcal F}}
\def\cM{{\mathcal M}}

\def\length{\operatorname{length}}

\def\vep{\varepsilon}

\def\TT{{\mathbb T}}
\def\RR{{\mathbb R}}

\def\ZZ{{\mathbb Z}}
\def\NN{{\mathbb N}}

\def\Diff{{\rm Diff}}
\def\Cl{{\rm Cl}}
\def\Gibb{{\rm Gibb}}
\def\inv{{\rm inv}}
\def\erg{{\rm erg}}
\def\Int{{\rm Int}}

\def\tak{{\tilde a_k}}

\title[DA diffeomorphisms]{Robust minimality of strong foliations for DA diffeomorphisms:  {$cu$}-volume expansion and new examples}
\author{Jana Rodriguez Hertz, Ra\'{u}l Ures and Jiagang Yang}
\date{\today}
\thanks{J.R.H. and R.U. are partially supported by NNSFC 11871262. J.R.H. is partially supported by NNSFC 11871394. J.Y. is partially supported by CNPq, FAPERJ, and PRONEX of Brazil and 
NNSFC 11871487 of China. Most of the research for this paper was made during a visit by J. Y. to SUSTech's Mathematics Department. J.Y. is very grateful for the good working environment during his visit and for the support received from the Department colleagues and authorities, in particular from J.R.H. and R.U. }

\address{1. Department of Mathematics, Southern University of Science and Technology, Shenzhen,
Guangdong, China}
\address{2. SUSTech International Center for Mathematics, Shenzhen,
Guangdong, China}
\email{rhertz@sustc.edu.cn}

\address{1. Department of Mathematics, Southern University of Science and Technology, Shenzhen,
Guangdong, China}
\address{2. SUSTech International Center for Mathematics, Shenzhen,
Guangdong, China}
\email{ures@sustc.edu.cn}

\address{Departamento de Geometria, Instituto de Matem\'atica e Estat\'istica, Universidade
Federal Fluminense, Niter\'oi, Brazil}
\email{yangjg@impa.br}

\begin{document}

\begin{abstract}
Let $f$ be a $C^2$ partially hyperbolic diffeomorphisms of $\TT^3$ (not necessarily volume preserving or transitive) isotopic to
a linear Anosov diffeomorphism $A$ with eigenvalues
$$\lambda_{s}<1<\lambda_{c}<\lambda_{u}.$$
Under the assumption that the set
$$\{x: \,\mid\log \det(Tf\mid_{E^{cu}(x)})\mid \leq \log \lambda_{u} \}$$
has zero volume inside any unstable leaf of $f$ where $E^{cu} = E^c\oplus E^u$ is the center unstable bundle, we prove that
the stable foliation of $f$ is $C^1$ robustly minimal, i.e., the stable foliation of
any diffeomorphism $C^1$ sufficiently close to $f$ is minimal. In particular, $f$  is robustly transitive.\par

We build, with this criterion, a new example of a $C^1$ open set of partially hyperbolic diffeomorphisms,
for which the strong stable foliation and the strong unstable foliation  are both minimal. 
\end{abstract}

\maketitle

\setcounter{tocdepth}{1}
\tableofcontents

\section{Introduction}

In this paper we will study the dynamics of certain types of partially hyperbolic diffeomorphisms as well as the properties of their invariant foliations. One of the most classically studied properties is transitivity.  
A diffeomorphism is  \emph{transitive} if it admits a dense orbit.
Transitivity is said to be {\em $C^r$ robust} (or stable) if it holds for every diffeomorphism $g$ in a $C^r$ neighborhood of $f$.

The first known examples of robustly transitive diffeomorphisms were the transitive Anosov diffeomorphisms.  This is a consequence of  their structural stability. It was not until the late 60's that nonhyperbolic robustly transitive  examples appeared. First it was Shub~\cite{Sh} who gave examples on $\mathbb T^4$. A few years later  Ma\~n\'e~\cite{M} presented a new class of examples on $\mathbb T^3$. These examples are called {\em DA diffeomorphisms} and are defined below in more detail. Ma\~n\'e's examples are strongly related to the results of this paper.  New advances in the study of robustly transitive diffeomorphisms occurred only in the 1990s.

Bonatti and D\'iaz \cite{BD} developed a new tool,  called blender, which made it possible to produce numerous new examples. For example,  they showed that some perturbations of certain products of Anosov diffeomorphisms (defined below), and  certain perturbations of the time one map of transitive Anosov flows are robustly transitive. 
All of these examples, including those from~\cite{Sh} and~\cite{M}, have a common property: they are partially hyperbolic. 
 A diffeomorphism $f$ of a closed manifold $M$ is {\em partially
hyperbolic} if the tangent bundle $TM$ splits into three invariant sub-bundles: 
$TM=E^{s}\oplus E^{c}\oplus E^{u}$
such that  all unit vectors
$v^\sigma\in E^\sigma(x)$ ($\sigma= s, c, u$) with $x\in M$ satisfy  :
\begin{equation}\label{pointwise.ph}
 |T_xfv^s| < |T_xfv^c| < |T_xfv^u| 
\end{equation}
for some suitable Riemannian metric. Here $T_xf$ is the tangent map of $f$ at the point $x$. The {\em stable bundle} $E^{s}$ must also satisfy
$\|Tf|_{E^s}\| < 1$ and the {\em unstable bundle}, $\|Tf^{-1}|_{E^u}\| < 1$. The bundle $E^{c}$ is called the {\em center bundle}.
Both bundles $E^{u}$ and $E^{s}$ are non-trivial. 
For further use, let us denote $E^{cu}=E^{c}\oplus E^{u}$ and $E^{cs}=E^{s}\oplus E^{c}$. When the center bundle $E^{c}$ is zero-dimensional, the diffeomorphism is called {\em Anosov} or {\em hyperbolic}. \newline\par

D\'iaz, Pujals and Ures \cite{DPU} proved that, in three-dimensional manifolds, robust transitivity implies partial hyperbolicity. However, the unstable bundle $E^{u}$ or the stable bundle $E^{s}$ could be trivial, though not both at the same time. See also \cite{BDP} for a higher dimensional version. \newline\par

It is well known  \cite{BP, HPS77} that there are two invariant foliations, the {\em stable} and {\em unstable} foliations, which are tangent, respectively, to $E^{s}$ and $E^{u}$. We refer to these foliations as the {\em strong foliations}. However, $E^{c}$ is not always integrable, not even in the case where it is one-dimensional. There are examples, even in 3-dimensional manifolds, of partially hyperbolic diffeomorphisms where the center bundle is not tangent to an invariant foliation, see \cite{HHUcoh, BGHP}. If there are invariant foliations $\cF^{i}$ with $i=sc, cu, c$ tangent, respectively, to the bundles $E^{i}$, then the diffeomorphism is called {\em dynamically coherent}. A non-dynamically coherent partially hyperbolic diffeomorphism is sometimes called {\em incoherent}. \newline\par

If the strong stable -or the strong unstable- foliation of a partially hyperbolic diffeomorphism is {\em minimal}, then
the diffeomorphism is transitive. Recall that a foliation is minimal if every leaf is dense in the whole manifold. The dynamical properties of these foliations are of great importance since they are intimately related to the dynamical properties of the diffeomorphism and of some of its most relevant invariant measures, such as physical measures, $u$-Gibbs measures, entropy maximizing measures, etc. Discussing the exact nature of this relationship is beyond the scope of this paper, but the interested reader may consult, for instance, \cite{D2, BFT}. \newline\par

There are  few results concerning  the robustness of the minimality of the strong stable foliations.   Bonatti,
D\'iaz and Ures~\cite{BDU} showed that either the strong stable or the strong unstable foliation is robustly minimal for three-dimensional robustly transitive diffeomorphisms. In \cite{HHUsome} there is also a version for partially hyperbolic diffeomorphisms in higher dimensions, with one-dimensional center bundle. 
Pujals and Sambarino~\cite{PS} proved that if each unstable leaf of $f$ contains a point whose $\omega$-limit set is uniformly hyperbolic, and if the diffeomorphism itself admits a minimal strong stable foliation, then  diffeomorphisms
in a $C^1$ neighborhood of $f$ also have minimal strong stable foliations.\newline\par

 In \cite{BDU} the robust minimality of both the strong stable and the strong unstable foliation is obtained by adding the following hypotheses: $f$ is dynamically coherent, all bundles $E^{s}$, $E^{c}$ and $E^{u}$ are orientable and $Tf(x)$ preserves their orientation, and there is a periodic compact center leaf. \newline\par

Up to now, all known partially hyperbolic diffeomorphisms for which both strong foliations are robustly minimal  fall into two categories: either they have a compact and periodic central curve as mentioned above, or they are dynamically incoherent.  Examples of the latter type were obtained in \cite{BGHP}. \newline \par

When the derivative $Tf(x)$ is expanding on the center bundle $E^{c}$ defined above, the diffeomorphism is Anosov. Even in this paradigmatic case, it is not known whether the strong unstable foliation is minimal.  This is an open question even for Anosov diffeomorphisms of the $3$-torus. Is the strong unstable manifold robustly minimal in this case? The numerical studies performed in \cite{GKM} suggest that the strong unstable manifold of the fixed point is dense for perturbations of a certain linear example. Furthermore, in \cite[Theorem 6.1]{HU2019} it is proved that for 3-dimensional Anosov diffeomorphisms, there is always a dense strong unstable leaf, though this leave is not necessarily the strong unstable manifold of the fixed point. A foliation is called {\em transitive} when it contains a dense leaf. 
\newline\par

In this work we will study the robust minimality of these foliations for {\em DA diffeomorphisms}, which we define as follows: let $A$ be a linear Anosov diffeomorphism  over $\TT^3$ with three distinct real eigenvalues. Then any $C^r$ ($r\geq 1$) partially hyperbolic diffeomorphisms in the isotopy class of $A$ will be called a DA diffeomorphism. We denote them by $\cD^r(A)$. In particular, Ma\~n\'e's construction in~\cite{M} leads to  DA diffeomorphisms.
The largest eigenvalue of the linear Anosov diffeomorphism $A$, denoted by $\lambda_{u}$, plays a key role in the study of ergodic measures of DA diffeomorphisms: for any DA diffeomorphism, its ergodic measures with entropy larger than $\log\lambda_{u}$
have the same structure as those of the linear Anosov diffeomorphism. Indeed, the semiconjugacy (for the precise definition, see Section~\ref{subsection.dynamical.coherence}) between the DA and the linear Anosov diffeomorphism $A$ is an isomorphism when restricted to this set of measures. The interested reader may consult \cite{VY2}. In this paper, we will further explain how $\lambda_{u}$ provides topological information about the diffeomorphism. We will also give examples for which both foliations are robustly minimal.

\subsection{Statement of the main result}
We will show that, for DA diffeomorphisms, if the volume along the $cu$-bundle has a non-uniform expansion (with respect to the constant $\log \lambda_{u}$),
then the stable foliation of this diffeomorphism is robustly minimal.

Although this result may seem somewhat unexpected as we do not assume transitivity of the original diffeomorphism, it was already conjectured by the second author in \cite{Ure12} (see the Introduction and Question 6.6 therein) that all DA diffeomorphisms are transitive. 
In some sense, the hypotheses in the previous results on the robust minimality of the  stable foliation is replaced here by the diffeomorphisms being
in an isotopy class of $A$ and having non-uniform volume  expansion on $E^{cu}$. This implies that every such 3-dimensional DA diffeomorphism really
admits some type of hyperbolic structure, which is mainly related to the constant $\log \lambda_{u}$. We are able to prove the previous conjecture in the case the diffeomorphism is sufficiently close  to one having enough  $cu$-expansion along unstable leaves. 
Here is our main result:

\begin{main}\label{m.boundedpotention} Let $A$ be a linear Anosov diffeomorphism of $\TT^3$ with eigenvalues
$$0<\lambda_{s}<1<\lambda_{c}<\lambda_{u}.$$
Let $f\in \cD^2(A)$ and suppose the set $\cB(f)=\{x: | \det(Tf\mid_{E^{cu}(x)}) | \leq \lambda_{u}\}$ has zero leaf volume
inside any strong unstable leaf, then the strong stable foliation of $f$ is robustly minimal.
\end{main}

\begin{remark}\label{rk.sharpbound}
Although the bound by $\log \lambda_{u}$ for the metric entropy of measures is sharp in~\cite{VY2} (see Propositions ~\ref{p.isomophic} and~\ref{p.noatom}),
our condition here
is not, since the hypothesis above does not hold under perturbations. However, the strong stable foliation remains robustly minimal.
\end{remark}

In the volume preserving setting the knowledge about these diffeomorphisms is more complete. Obviously, transitivity in this situation is easier to obtain. In case the  DA diffeomorphism is conservative, transitivity is a consequence of the results of \cite{HU}. There is even  a more complete description as Gan and Shi \cite{GS} have shown that these diffeomorphisms are ergodic. 

\subsection{Structure of the proof of Theorem \ref{m.boundedpotention}} Usually, when proving that a stable foliation is robustly minimal, 
the proof, roughly speaking, can be divided into two steps:
\begin{enumerate}
	\item[(a)] for every open set $U$ the forward iteration of $U$ contains a set that has uniform size along the center-unstable direction;
	\item[(b)] every (strong) stable leaf must intersect this set.
\end{enumerate}

In~\cite{BDU} this is achieved by showing the existence of an {$s$-section}, that is, a two-dimensional surface that transversally intersects every stable leave. 
In~\cite{PS} the authors proved the step (a) under the assumption that for every $x\in M$, there is a point $y$ in $\cF^u(x)$ whose center bundle is uniformly expanding. 
In the meantime, $f$ having a minimal strong stable foliation implies that the strong stable leaves for every nearby $C^1$ diffeomorphism $g$ are $\varepsilon$-dense, thus satisfying the step (b).

The proof in our paper also follows this path, albeit with a completely different technique:

To achieve the step (a), we mainly deal with the $cu$-bundle. We show that a diffeomorphism under the assumptions of Theorem~\ref{m.boundedpotention} has {\em mostly expanding center} (Section~\ref{s.ME}). Thus, by a general technique for diffeomorphisms with mostly expanding center introduced in Section~\ref{s.theoryofME},  for each nearby $C^1$ diffeomorphism, the forward orbit of Lebesgue almost every point is eventually expanding along the center direction. See \cite{Y1} for more details.

The main tool to study diffeomorphisms with mostly expanding center is a subset of the invariant  probability measures which we denote by $G(f)$. To be more precise, we consider the following set of invariant measures:
\begin{equation*}
\G^u(f)=\{\mu\in \cM_{inv}(f): h_\mu(f,\cF^u)\geq \int \log(\det(Tf\mid_{E^u(x)}))d\mu(x)\}.
\end{equation*}
$$
\G^{cu}(f)=\{\mu\in \cM_{\inv}(f): h_\mu(f)\geq \int \log(\det(Tf\mid_{E^{cu}(x)}))d\mu(x)\}
$$
and consider their intersection:
\begin{equation*}
\G(f)=\G^{u}(f)\cap \G^{cu}(f). 
\end{equation*}
The definition of $h_\mu(f,\cF^u)$, the partial entropy along the foliation $\cF^u$, can be found in Section~\ref{ss.2.3}. Here we collect some properties of $G^u(f), G^{cu}(f)$ and $G(f)$ which will be important to us. The precise statements of those properties can be found in Section~\ref{s.theoryofME}. For any $C^1$ partially hyperbolic diffeomorphisms $f$  (see~\cite{HYY}):
\begin{enumerate}
\item 

The set $G ^u(f)$ is non-empty, convex and compact.
If $f$ is $C^2$ then $G^u(f)$ is the set of Gibbs $u$-states of $f$. 
$G^u(f)$  varies upper semi-continuously in $C^1$ topology: if $f_n\xrightarrow{C^1} f$ then $\limsup_n G^u(f_n)\subset G^u(f)$.
$G^u(f)$ contains all the weak-* limit points  of the Ces\'aro  limit at Lebesgue almost every $x$. And 
the extremal elements of $G^u(f) $ are ergodic. 
\item Furthermore, the set $G^{cu}(f)$ 
is non-empty and convex and contains all the weak-* limit points  of the Ces\'aro  limit at Lebesgue almost every $x$.
\item Finally, if $f$ has mostly expanding center, then:

 $G(f)$ is compact and contains all the physical measures of $f$. The
 extremal elements of $G(f) $ are ergodic. 
The measures in $G(f)$ have only positive center exponents and satisfy Pesin's entropy formula. 
 $G(f)$ varies upper semi-continuously in the $C^1$ topology.
\end{enumerate}
The last two properties of $G(f)$ allow us to apply the Pliss Lemma and hyperbolic times argument to nearby $C^1$ diffeomorphisms. This shows that the local unstable manifolds at certain iterations along typical orbits must have uniform size. See Lemma~\ref{l.hyperbolictime}.


The proof of the step (b) is more involved. We have to deal with the strong stable foliation without any assumption on minimality or even  transitivity of $f$. We show that, for any $C^1$ DA diffeomorphism, and any ergodic measure with positive center exponent and with entropy larger than $\log \lambda_{u}$, the  Pesin unstable manifolds at regular points, whose existence is given by the Pesin's theory, coincide with $\cF^{cu}(x).$
The Pesin unstable manifold of a regular point $x\in M$ consists of the points $y\in M$ such that 
$$\lim_{n\to\infty}d(f^{-n}(x),f^{-n}(y))=0. $$ The Pesin unstable manifold $W^{u}(x)$ is dynamically defined, whereas $\cF^{{cu}}(x)$ is topologically defined. In general $W^{u}(x)\subset \cF^{cu}(x)$, but they are not necessarily equal. In Proposition \ref{p.sizeofstablemanifold} we prove that for certain points $x$ both manifolds coincide. 
This part heavily uses the results in~\cite{VY2} on the classification of measures with large entropy for DA diffeomorphisms. Since the lift of $\cF^{cu}$ and $\cF^s$ to the universal covering space form a product structure, it follows that the Pesin unstable manifolds are $s$-sections, intersecting every stable leaf.

\subsection{New examples}
Let us observe that the previous method only works for the minimality of the strong stable foliation, the
minimality of the strong unstable foliation is still an open question, even when the partially hyperbolic diffeomorphism is Anosov. But with the criterion above we may provide the following:

\begin{main}\label{main.example}
There exist DA diffeomorphisms on $\mathbb T^3$ such that both strong stable and unstable foliations are robustly minimal. 
\end{main}

The construction is quite different from the one used by R. Ma\~n\'e in \cite{M}. Ma\~n\'e carefully performed the perturbation near a fixed point 
so that the perturbed diffeomorphism remains partially hyperbolic. Hence the modification is mainly supported inside a small ball, and the non-hyperbolicity is local. In our construction, we need to modify the dynamics of the linear Anosov diffeomorphism
in a cylindrical neighborhood of a long center segment which could eventually become $\vep$-dense for a small $\vep>0$. Therefore we need to carefully choose the linear Anosov diffeomorphism, and a subtle analysis is required. Indeed, we will choose a sequence of linear Anosov automorphisms $A_k$ so that their center Lyapunov exponent converge to zero. Such a sequence has been firstly considered in \cite{PT}. In this manner we obtain a new way of reaching the boundary of the set of Anosov diffeomorphisms, at least in the isotopy classes of $A_k$ for $k$ sufficiently large.
\par
Can we perform this construction in the isotopy class of any hyperbolic automorphism of the 3-torus? Concretely, we have the following: 

\begin{conjecture}
Let $A$ be a linear Anosov diffeomorphism  of $\TT^3$ with eigenvalues
$$0<\lambda_{s}<1<\lambda_{c}<\lambda_{u}.$$
There is a $C^1$ open set of partially hyperbolic diffeomorphisms $\cU$ isotopic to $A$, such that
for any diffeomorphism $g\in \cU$, both its strong stable and unstable foliation are minimal.
\end{conjecture}

\subsection{Structure of the paper}
This paper is organized in the following way.

In Section~\ref{s.prelinimaries} we give some necessary background.

In Section~\ref{s.theoryofME} we go beyond the scope of DA diffeomorphisms, and introduce a general theory for diffeomorphisms with mostly expanding center direction. In particular, 
we provide the main tool for the study of such diffeomorphisms: a special space of probability measures, denoted by $\G(f)$, 
which is defined using the partial entropy along unstable leaves. See Definition~\ref{df.G} and~\eqref{eq.G(f)}.  Then we verify that the diffeomorphisms we are considering have mostly expanding center. This is carried out in Section~\ref{s.ME}, Lemma~\ref{lemma.positive}.

Section~\ref{s.hyperbolicity} consists of a different theory which only applies for  DA diffeomorphisms.
 We show that any ergodic probability measure with large entropy is hyperbolic, and moreover,
the unstable manifold of a typical point coincides with the corresponding center-unstable leaf of the partially hyperbolic diffeomorphism. See Proposition~\ref{p.sizeofstablemanifold}.

Theorem~\ref{m.boundedpotention} is proven in Section~\ref{s.sminimal}, 
and in Section~\ref{s.examples} we build the examples of Theorem~\ref{main.example}.

\section{Preliminaries\label{s.prelinimaries}}
In this section we introduce the necessary background for the proofs.
Throughout this section we keep the same hypothesis as in the first section; that is, $A$ is
a three dimensional linear Anosov diffeomorphism with eigenvalues $0<\lambda_{s}<1<\lambda_{c}<\lambda_{u}$.

\subsection{Dynamical coherence} \label{subsection.dynamical.coherence}
A partially hyperbolic diffeomorphism is said to be {\em dynamically coherent} if it admits invariant foliations $\cF^i$, $i=cs,cu, c$, tangent to the corresponding bundles at each point. \par

By Franks~\cite{Fra70}, for every DA diffeomorphism $f\in\cD^1(A)$, there exists a continuous surjective map $\phi: \TT^3\to \TT^3$
that semi-conjugates $f$ to $A$, that is, $\phi \circ f=A\circ\phi$. The following properties of $\phi$ hold for $f\in\cD^{1}(A)$:

\begin{proposition}\label{p.coherent}
Suppose $f\in \cD^1(A)$. Then $f$ is dynamically coherent, and the Franks' semi-conjugacy $\phi$ maps the strong stable, center stable, center, and center unstable leaves of $f$ into the corresponding leaves of $A$. Moreover,
\begin{itemize}
\item[(a)] $\phi$ restricted to each strong stable leaf is bijective;

\item[(b)] there is $K>0$ only depending on $f$, such that for every
$x\in \TT^3$, $\phi^{-1}(x)$ is either a point, or a connected segment inside a center leaf with length bounded
by $K$.

\end{itemize}
\end{proposition}
\begin{proof}
By Potrie~\cite{Pot15}, $f$ is dynamically coherent. Items (a) and (b) are proven in \cite{Ure12} (see also \cite[Proposition 3.1]{VY2} for a proof of (b)).
\end{proof}

As a consequence of (b), we have:
\begin{corollary}\label{c.preserceentropy}
For any $f\in \cD^1(A)$, $\phi$ preserves  the metric entropy; that is, for any invariant measure $\mu$ of $f$:
$$h_\mu(f)=h_{\phi_*\mu}(A).$$
\end{corollary}

\begin{proof}
For every $x\in\TT^3$, one has 
$$
f(\phi^{-1}(x))  = \phi^{-1}(Ax )
$$
which has length bounded by $K$, and can be covered by no more than $2K/\beta$ many $\beta$-balls. For each $\beta>0$, denote by $r_n(x,\beta)$ the minimum cardinality of $(n,\beta)$-spanning sets of $\phi^{-1}(x)$. Then 
$$
r_n(x,\beta) \le 2Kn/\beta.
$$
This shows that $h_{top}(f, \phi^{-1}(x)) = 0$ for all $x\in\TT^3$. Then the corollary follows from the Ledrappier-Walters' formula \cite{LW}.
\end{proof}

Denote by $\tilde{f}$ the lift of $f$ to the universal covering space $\mathbb{R}^{3}$ and by $\tilde{\cF}^i_{\tilde{f}}$ ($i=s,c,u,cs,cu$) the lift of the corresponding foliations of $f$ to the universal covering space.

\begin{definition}\label{d.productstructure}
We say $f$ (or $\tilde{f}$) has \emph{global product structure} if for any two points $\tilde{x},\tilde{y}\in \mathbb{R}^3$,
$\tilde{\mathcal{\cF}}^{u}_{\tilde{f}}(\tilde{x})\bigcap \tilde{\mathcal{\cF}}^{cs}_{\tilde{f}}(\tilde{y})$ consists of a unique point, and
$\tilde{\mathcal{\cF}}^{cu}_{\tilde{f}}(\tilde{x})\bigcap \tilde{\mathcal{\cF}}^{s}_{\tilde{f}}(\tilde{y})$ consists of a unique point.
\end{definition}

The following result was proved in \cite[Proposition 7.3]{Pot15}:
\begin{proposition}\label{p.product}
Every $f\in \cD^1(A)$ has global product structure.
\end{proposition}

As a direct corollary we have that:
\begin{corollary}\label{c.csproduct}
For any $x\in \TT^3$, $\cF^{cu}_f(x)=\bigcup_{y\in \cF^c_f(x)}\cF^u_f(y)$.
\end{corollary}

Moreover, in Lemma \ref{l.uniformintersection} below we obtain a uniform control on the global product structure. A {\em fundamental domain} of the lift of $\TT^{3}$ to its universal cover is a closed set $B$ with $B = \Cl(\Int(B))$, such that its image under the covering map $\pi:\RR^{3}\to\TT^{3}$ is $\TT^{3}$, and such that $\pi$ is injective when restricted to the interior of $B$. 

For $R>0$, we define $\tilde \cF^\sigma_R(x)$ to be the $R$-ball inside the leaf $\tilde \cF^\sigma(x)$, $\sigma = s,c,u,cs,cu$.
\begin{lemma}\label{l.uniformintersection}
For any $f\in \cD^1(A)$, any $\tilde{x}\in \mathbb{R}^3$, and any fundamental domain $\tilde{\TT}^3$ of the lift
of $\TT^3$, there is $R_{\tilde{x}}>0$ such that for any $\tilde{y}\in \tilde{\TT}^3,$ $\tilde{\cF}^s(\tilde{y})\pitchfork \tilde{\cF}^{cu}_{R_{\tilde{x}}}(\tilde{x})\neq \emptyset$.
\end{lemma}
\begin{proof}
By Proposition~\ref{p.product}, for any $\tilde{x},\tilde{y}\in\RR^{3}$, there exists $R(\tilde{y})> 0$ such that $\tilde{\cF}^s(\tilde{y})\pitchfork \tilde{\cF}^{cu}_{R(\tilde{y})}(\tilde{x})\neq \emptyset$.
Because both foliations $\tilde{\cF}^s$ and $\tilde{\cF}^{cu}$ vary continuously with respect to points, the previous intersection
still holds for points in a neighborhood $U_{\tilde{y}}$ of $\tilde{y}$. Since $\tilde{\TT}^3$ is compact, we may take a finite open cover $U_{i}$ $(i=1,\cdots, k)$ of $\tilde{\TT}^3$, and take $R_{\tilde{x}}=\max \{R_i\}_{i=1}^k$. The proof is complete.
\end{proof}

 \begin{remark}\label{r.uniformintersection}
By projecting on ${\TT}^3$, it trivially follows from the lemma above that for any $x\in {\TT}^3$,  there is $R_{x}>0$ such that for any $y\in\TT^3$, $\cF^s(y)\pitchfork \cF^{cu}_{R_{x}}(x)\neq \emptyset$.

\end{remark} 

\subsection{Measure-theoretical information} Given any diffeomorphism $f:M\to M$, let $\cM_{\inv}(f)$ be the set of $f$-invariant probability measures on $M$, and $\cM_{\erg}(f) \subset \cM_{\inv}(f)$ be the set of ergodic probability measures.\par
Let $\phi:\TT^{3}\to\TT^{3}$ be the semiconjugacy defined in Subsection \ref{subsection.dynamical.coherence}. 
Since $\phi$ is not injective in general, the map
$$\phi_*:\cM_{\inv}(f)\to \cM_\inv(A)$$
is usually not injective either. Surprisingly, it is proven in~\cite[Theorem 3.6]{VY2} that if one restricts $\phi_*$ to the set of ergodic measures with large entropy:
$$\phi_*:\{\mu\in \cM_{\erg}(f): h_\mu(f)>\log \lambda_{u}\}\to \{\nu\in\cM_\erg(A): h_\nu(A)>\log \lambda_{u}\},$$
then it is bijective.
Moreover, the following two properties were proved in \cite{VY2}, showing that the constant $\log \lambda_{u}$
is important to classify ergodic measures with large entropy. We should note that, as explained in \cite{VY2},
the constant $\log \lambda_{u}$ is sharp here.

\begin{proposition}\cite[Theorem 3.6]{VY2} \label{p.isomophic}
Let $f\in\cD^1(A)$ and $\mu$ be an ergodic probability measure $\mu$ of $f$ with $h_\mu(f)> \log \lambda_{u}$.
Then for $\mu$ almost every $x$, $\phi^{-1}\circ \phi(x)=\{x\}$; that is,
$\phi$ is almost surely bijective on the support of ergodic measures with entropy larger than $\log \lambda_{u}$.
\end{proposition}

In Proposition \ref{p.isomophic}, the support of the above mentioned ergodic measures could be {\em a priori} closed $f$-invariant sets that are not the whole manifold. Moreover, it is not known in general whether a DA diffeomorphism is necessarily transitive. 

\subsection{Partial entropy along an expanding foliation}\label{ss.2.3} Let $\mu$ be a probability measure and let $\xi$ be a measurable partition of $M$. Then for $\mu$-almost every $x\in M$ there exists a conditional probability measure $\mu_{x}^{\xi}$ such that for any measurable set $A$, the map $x\mapsto \mu_{x}^{\xi}$ is measurable and
$$\mu(A)=\int \mu_{x}^{\xi}(A)d\mu(x).$$
The collection of the conditional measures $\{\mu^\xi_x\}$ is called the {\em disintegration of $\mu$ with respect to $\xi$.}
See \cite{Rok49} for more details. The conditional measure $\mu^{\xi}_{x}$ actually depends only on the element of the partition $\xi(x)$ that contains $x$, so we may sometimes denote it by $\mu^{\xi}_{\xi(x)}$. \par
Let $C_{1},C_{2},\dots$ be the elements of the partition $\xi$ that have positive $\mu$-measure (if any). Then the {\em entropy of the partition $\xi$} is defined by
$$H_{\mu}(\xi)=-\sum_{n} \mu(C_{n})\log\mu(C_{n})$$
if $\mu(M\setminus\bigcup_{n} C_{n})=0$ or by $H_{\mu}(\xi)=\infty$ otherwise. \par
Given two measurable partitions $\xi$ and $\eta$, for every set $B\in\eta$ the partition $\xi$ induces a partition $\xi_{B}$ of $B$. The {\em mean conditional entropy of $\xi$ with respect to $\eta$} is defined by
$$H_{\mu}(\xi|\eta)=\int_{M/\eta} H_{\mu^{\eta}_{B}}(\xi_{B})d\mu_{\eta}(B),$$ 
where $M/\eta$ is the quotient of $M$ by the partition $\eta$, $\mu^{\eta}_{B}$ is the probability $\mu^{\eta}_{x}$ for any $x\in B$, and $\mu_{\eta}$ is the quotient measure. 

For two partition $\xi$ and $\eta$, we write $\xi\prec \eta$ if $\eta$ is finer than $\xi$, i.e., if every element of $\eta$ is contained in some element of $\xi$. We also denote by $\xi\vee\eta$ their join, i.e., the partition formed by taking the intersection between their elements. 
If $\{\eta_{n}\}_{n=1}^{\infty}$ is a sequence of measurable partitions, we will denote $\eta_{n}\uparrow\eta$ if
\begin{itemize}
\item $\eta_{1}\prec\eta_{2}\prec\dots$, 
\item $\bigvee_{n=1}^{\infty}\eta_{n}=\eta$. 
\end{itemize}
If $\xi$ and $\eta_{n}$ are measurable partitions for $n=1,2,\dots$ such that $\eta_{n}\uparrow\eta$ and $H_{\mu}(\xi|\eta_{1})<\infty$, then 
$$H_{\mu}(\xi|\eta_{n})\downarrow H_{\mu}(\xi|\eta).$$
See for instance \cite[Subsection 5.11]{Rok67}. Given any measurable partition $\xi$, we define
$$h_{\mu}(f,\xi)=H_{\mu}(\xi\mid \bigvee_{n=1}^{\infty}f^{n}\xi)$$
 A measurable partition $\xi$ is {\em increasing} if $f\xi\prec\xi$. If a partition $\xi$ is increasing, then 
$$h_{\mu}(f,\xi)=H_{\mu}(\xi|f\xi).$$

An $f$-invariant foliation $\cF$ is {\em expanding} if the derivative of $f$ along $\cF$ is uniformly expanding. Given an invariant probability measure $\mu$ and an invariant expanding foliation $\cF$ of $f$, a measurable partition $\xi$ is {\em $\mu$-subordinate} to the foliation $\cF$ if for $\mu$-almost every $x$ it satisfies
\begin{enumerate}
 \item $\xi(x)\subset\cF(x)$ and $\xi(x)$ has a uniformly bounded from above diameter inside $\cF(x)$, 
 \item $\xi(x)$ contains an open neighborhood of $x$ inside the leaf $\cF(x)$, 
 \item $\xi$ is an increasing partition.
\end{enumerate}

Such a partition always exists if $\cF$ is expanding, see \cite{LS82} and \cite{Y0}.

Given an invariant probability $\mu$, an invariant expanding foliation $\cF$, and any $\mu$-subordinate measurable partition $\xi$, we define the \emph{partial entropy of $f$ along $\cF$}, which we denote by $h_\mu(f,\cF)$, as:
$$
h_\mu(f,\cF) = h_\mu(f,\xi).
$$
It is well known that the definition above does not depend on the choice of the partition $\xi$, see for example~\cite{LY85a}.

\begin{proposition}\cite[Proposition 2.7]{VY2}\label{p.noatom}
	Let $\cF$ be an expanding foliation, $\mu $ be an ergodic probability measure, and $\{\mu_x : x \in M\}$ be the disintegration of $\mu$ with respect to any measurable partition $\xi$ that is $\mu$-subordinate to $\cF$. Then the following conditions are equivalent:\begin{enumerate}
		\item $h_\mu(f, \cF) > 0;$
		\item for $\mu$-almost every point $x$, the measure $\mu_x$ is continuous; that is, it has no atoms.
	\end{enumerate}
	Moreover,   if {\rm (1)} (or {\rm (2)}) is satisfied, any full $\mu$-measure subset $Z$ intersects almost every leaf of $\cF$ at an uncountable set.

\end{proposition}

The definition of expanding foliation does not require the foliation to be the strong unstable foliation. A particularly important case is the center foliation $\mathcal F^c_A$ of the linear Anosov diffeomorphism $A$. This foliation is an expanding foliation, which means that its partial entropy is well defined. We will need the following proposition, which is a direct consequence of \cite[Proposition 2.8]{VY2}. 

\begin{proposition}\cite[Proposition 2.8]{VY2}\label{p.centerlinear}
Let $A$ be a linear Anosov diffeomorphism of $\TT^3$ with eigenvalues
$0<\lambda_{s}<1<\lambda_{c}<\lambda_{u}$ and $\nu$ an $A$-invariant measure with $h_\nu(A)>\log \lambda_{u}$. Then
$$h_\nu(A,\cF^c_A)\geq h_{\nu}(A)-\log \lambda_{u}>0.$$

\end{proposition}

%

\section{Diffeomorphisms with mostly expanding center\label{s.theoryofME}}
Our proof depends heavily on measure-theoretical arguments. In this section we are
going to precisely state the properties of $G^u(f), G^{cu}(f) $ and $G(f)$. We will also collect some basic background on diffeomorphisms with mostly expanding center, introduced in \cite{ABV,Y0,Y1}.


In this section $f$ will be a $C^1$ partially hyperbolic diffeomorphism unless otherwise specified. It is well-known that there is an invariant foliation $\cF^{u}$ tangent to $E^{u}$ which is expanding, see for instance \cite{BP, HPS77}. 


\subsection{$G^u$ states} 
First, recall that for a $C^{2}$ partially hyperbolic diffeomorphism, following the work of Pesin and Sinai~\cite{PS82}, a \emph{Gibbs $u$-state} is an invariant probability
measure whose conditional probabilities (in the sense of Rokhlin~\cite{Rok49}) along strong unstable leaves
are absolutely continuous with respect to the Lebesgue measure on the leaves. 
The set of Gibbs $u$-states plays an important role in the study of physical measures for $C^2$
partially hyperbolic diffeomorphisms. More properties for Gibbs $u$-states can be found in
\cite[Subsection 11.2]{BDVnonuni}, see also \cite{D2, D}.

Below we are going to define a natural generalization of Gibbs $u$-states for $C^1$ partially hyperbolic
diffeomorphisms.

\begin{definition}\label{df.Gu}
We define:
\begin{equation}\label{eq.Gu}
\G^u(f)=\{\mu\in \cM_{inv}(f): h_\mu(f,\cF^u)\geq \int \log(\det(Tf\mid_{E^u(x)}))d\mu(x)\}.
\end{equation}
\end{definition}

\begin{remark}\label{rk.Gu}
\begin{itemize}
\item[(a)] When $f$ is $C^{2}$, by Ledrappier~\cite{L84}, $\G^u(f)$ is the set of  Gibbs $u$-states of $f$.
\item[(b)] By the Ruelle's inequality for partial entropy (see for instance \cite{WWZ}), one can replace the inequality in the definition of $\G^u$ by equality:
$$\G^u(f)=\{\mu\in \cM_{inv}(f): h_\mu(f,\cF^u)= \int \log(\det(Tf\mid_{E^u(x)}))d\mu(x)\}.$$
\end{itemize}
\end{remark}

The following property that was already known  for $\Gibb^u(f)$, also holds for $\G^u(f)$.
\medskip

\begin{proposition}~\cite[Propositions 3.1, 3.5]{HYY}\label{p.Gu}
For any $C^1$ partially hyperbolic diffeomorphism $f$, $\G^u(f)$ is non-empty, convex, compact, and varies in an upper semi-continuous way with respect to the partially
hyperbolic diffeomorphisms endowed with the $C^1$ topology. Moreover, for any invariant measure $\mu\in \G^u(f)$, almost every ergodic component of
its ergodic decomposition  belongs to $\G^u(f)$.
\end{proposition}

\subsection{Other invariant measure subspaces\label{ss.candidates}}

\begin{definition}\label{df.G}
 Let $E^{cu}=E^{c}\oplus E^{u}$. Define
 $$
\G^{cu}(f)=\{\mu\in \cM_{\inv}(f): h_\mu(f)\geq \int \log(\det(Tf\mid_{E^{cu}(x)}))d\mu(x)\}
.$$

\end{definition}

Note that $\G^{cu}(f)$ is defined similarly to $\G^u(f)$, but using the metric entropy $h_\mu(f)$ instead of the partial entropy. Also note that measures in $\G^{cu}(f)$ may have negative center exponent. In fact, if $\mu\in\G^u(f)$ has only negative center exponents, then 
$$h_{\mu}(f)\geq h_{\mu}(f,\cF^{u})\geq \int \log(\det(Tf\mid_{E^u(x)}))d\mu(x)\geq \int \log(\det(Tf\mid_{E^{cu}(x)}))d\mu(x),$$ 
and therefore
$\mu$ must belong to $\G^{cu}(f)$.

Finally, we denote
\begin{equation}\label{eq.G(f)}
\G(f)=\G^{u}(f)\cap \G^{cu}(f). 
\end{equation}

The most important property of the space $G(f)$ is that it contains all physical measures of $f$. 
An invariant probability
measure $\mu$ of $f$ is a {\em physical measure} if its basin
$\cB(\mu)=\{x; \lim_{n\to \infty}\frac{1}{n}\sum_{i=0}^{n-1}\delta_{f^i(x)}=\mu\}$
has positive volume. The limit in the definition of the set $\cB(\mu)$ corresponds to the weak$^*$ topology. The measure $\delta_{x}$ denotes the Dirac probability measure supported at $x$.

In the proposition below it is shown that $G(f)$ is always non-empty and that the space $\G(f)$ contains all the physical measures. However the space of physical measures could be empty in general. 

\begin{proposition}\cite[Theorem A]{HYY}\cite[Proposition 2.12]{Y0}\label{p.physical} Let $f$ be a $C^{1}$ partially hyperbolic diffeomorphism. Then 
there is a full volume subset $\Gamma$ such that for any $x\in\Gamma$, any limit of the
sequence $\frac{1}{n}\sum_{i=0}^{n-1}\delta_{f^i(x)}$ belongs to $\G(f)$.
\end{proposition}

In general, the structure of $\G(f)$ is not as clear as that of $\G^u(f)$; for instance, it is not
always true that the extreme elements of $\G(f)$ are all ergodic. This is due to the presence of measures with negative center exponent. 
This changes under some extra hypotheses, see next Subsection.

\subsection{Diffeomorphisms with mostly expanding center}
From now on, we will assume that $f$ is $C^{2}$. 

Diffeomorphisms with mostly expanding center were introduced by Alves, Bonatti, and Viana \cite{ABV}
 using a technical definition regarding the Lyapunov exponents on sets with positive leaf volume. Later, a narrower definition was given
by Andersson and  V\'asquez \cite{AV}. The two definitions are not equivalent. Here we use the definition of \cite{AV}. 

\begin{definition}\label{df.me}
	$f$ has mostly expanding center if all the center exponents of every Gibbs $u$-state of $f$ are positive.
\end{definition}

\begin{proposition}\cite[Proposition~5.17]{Y1}\label{p.meGspace}
Suppose $f$ is a $C^2$ partially hyperbolic diffeomorphism with mostly expanding center, then there is a $C^1$ neighborhood $\cU$ of $f$ such that, for
any $C^1$ diffeomorphism $g\in \cU$, $\G(g)$ is compact and convex, and every extreme element of $\G(g)$
is an ergodic measure. In particular, almost every ergodic component of $\nu \in\G(g)$ is also in $\G(g)$.

  Moreover, the map
$\cG : g \mapsto \G(g)$ restricted to $\cU$ is upper semi-continuous
under the $C^1$ topology.
\end{proposition}


If we combine the definition of the space $\G^{cu}(f)$ and Ruelle's inequality, we have that:
\begin{corollary}\label{c.mePesinformula}
Suppose $f$ has mostly expanding center, then there is a $C^1$ neighborhood $\cU$ of $f$, such that for
any $C^1$ diffeomorphism $g\in \cU$, every probability $\mu\in \G(g)$ satisfies Pesin formula:
$$h_\mu(g)=\int \log(\det(Tg\mid_{E_g^{cu}(x)}))d\mu(x)=\sum_{\lambda_i(\mu,g)>0}\lambda_i(\mu,g).$$
\end{corollary}

Later, we will use this corollary to obtain a lower bound of the metric entropy for measures in $\G(g)$, which will enable us to apply Propositions~\ref{p.isomophic} and~\ref{p.noatom}.

\section{Positive center exponent\label{s.ME}}
Throughout this section let $f$ be a $C^2$ partially hyperbolic diffeomorphism satisfying the assumptions of  Theorem~\ref{m.boundedpotention}.
We are going to show that $f$ has mostly expanding center.

By the assumption of Theorem~\ref{m.boundedpotention}, the set $\cB(f)=\{x: | \det(Tf\mid_{E^{cu}(x)}) | \leq \lambda_{u}\}$ has zero leaf volume inside any strong unstable leaf.  Therefore, since the conditional measures of Gibbs $u$-states along the unstable leaves are equivalent to the Lebesgue measures on the
corresponding leaves,  we have
\begin{lemma}\label{l.cuexponents}
	For any Gibbs $u$-state $\mu$ of $f$,
	\begin{equation}\label{eq.cuexponents}
	\lambda^u(\mu,f)+\lambda^c(\mu,f)=\int \log |\det(Tf\mid_{E^{cu}(x)})| d\mu(x)> \log \lambda_{u}.
	\end{equation}
\end{lemma}

More importantly, we obtain the uniform positivity for the center Lyapunov exponent:
\medskip

\begin{lemma}\label{lemma.positive} The  diffeomorphism
	$f$ has mostly expanding center. That is, for any Gibbs $u$-state $\mu$ of $f$, the center Lyapunov exponent of $\mu$, $\lambda^c(\mu,f)$,
	is positive:
	\begin{equation}\label{eq.positiveexponent}
	\lambda^c(\mu,f)=\int \log |\det(Tf\mid_{E^{c}(x)})| d\mu(x)>0.
	\end{equation}
\end{lemma}
%
%
%

The proof of this lemma follows from the following theorem:
\begin{theorem} \cite[Theorem A]{VY2}\label{thm.positive} Let $\mu$ be an ergodic invariant probability measure of $f\in {\mathcal D}^{1}(A)$ with
$h_{\mu}(f)>\log \lambda_{u}$. Then every full $\mu$-measure set $Z\subset M$ intersects almost every
center leaf in an uncountable subset. Moreover, the center Lyapunov exponent
along the center direction is non-negative, and even strictly positive if f is $C^{2}$. 
\end{theorem}

\begin{proof}[Proof of Lemma \ref{lemma.positive}] Suppose that $\lambda^c(\mu,f)\leq 0$. Since $\mu$ is a Gibbs $u$-state, it has  an ergodic component $\mu'$ such that $\lambda^c(\mu',f)\leq 0$ and it is also a Gibbs $u$-state (see Proposition \ref{p.Gu}). As a consequence of Lemma \ref{l.cuexponents}, we have that $h_{\mu'}(f)> \log \lambda_{u}$. Then we can apply Theorem \ref{thm.positive} above  to $\mu'$ and since $f$ is $C^2$, we arrive to a contradiction.  We therefore get that the center exponent is positive. 
\end{proof}  

Observe that the integration in \eqref{eq.cuexponents} and \eqref{eq.positiveexponent} depends continuously on the measures since the integrand is a continuous function. On the other hand, the space of Gibbs $u$-states is compact. As a result, there is $a>0$ such that for any Gibbs $u$-state $\mu$ of $f$,
we have:

\begin{equation}\label{eq.uniformcuexponents}
\lambda^u(\mu,f)+\lambda^c(\mu,f)=\int \log |\det(Tf\mid_{E^{cu}(x)})| d\mu(x)> \log\lambda_{u}+2a,
\end{equation}

and
\begin{equation}\label{eq.uniformcexponent}
\lambda^c(\mu,f)=\int \log |\det(Tf\mid_{E^{c}(x)})| d\mu(x)>2a.
\end{equation}

(\eqref{eq.uniformcuexponents} follows from \eqref{l.cuexponents} and \eqref{eq.uniformcexponent} follows from \eqref{eq.positiveexponent}, we then take a minimum $a>0$).

By Proposition~\ref{p.meGspace} and Corollary~\ref{c.mePesinformula} we also get:
\begin{lemma}\label{l.robustuniform}
	There is a $C^1$ neighborhood $\cU$ of $f$, and $a>0$ such that for any $C^1$ diffeomorphism $g\in \cU$,
	and any invariant measure $\nu \in \G(g)$,
	$$h_\nu(g)=\lambda^u(\nu,g)+\lambda^c(\nu,g)=\int \log |\det(Tg\mid_{E^{cu}(x)})| d\nu(x)> \log \lambda_{u}+a,$$
	$$\text{and  }\;\;\;\;\;\; \lambda^c(\nu,g)=\int \log |\det(Tg\mid_{E^{c}(x)})| d\nu(x)>a.$$
\end{lemma}

\begin{proof}
By Proposition \ref{p.meGspace}, $G(f)$ is compact. Since $\int \log|\det (Tf|_{E^{cu}(x)})|d\mu(x)$ and $\int \log|\det (Tf|_{E^{c}(x)})|d\mu(x)$ depend continuously on $\mu$ and $f$, \eqref{eq.uniformcuexponents} and \eqref{eq.uniformcexponent} imply that  there exist open neighborhoods $\mathcal{V}$ of  $G(f)$ in the space of probability measures and $\mathcal U$ of $f$ such that,  for $\nu\in \mathcal{V}$ and $g\in\mathcal{U}$, we have $\int \log|\det (Tg|_{E^{cu}(x)})|d\nu(x)>\log \lambda_{u}+a$ and $\int \log|\det (Tg|_{E^{c}(x)})|d\nu(x)>a$. The upper semicontinuity of the function $g\mapsto G(g)$ (Proposition~\ref{p.meGspace}) implies that $G(g)\subset \mathcal{V}$ for every $g\in \mathcal U$ if $\mathcal{U}$ is small enough. This gives the two inequalities of the lemma. The Pesin formula (Corollary~\ref{c.mePesinformula}) gives the first equality of the first equation.  
\end{proof}


\section{Ergodic measures with large entropy\label{s.hyperbolicity}}

Throughout this section let $f$  be a $C^1$ DA diffeomorphism and let $\mu$ be an $f$-invariant probability measure such that
\begin{enumerate}
 \item $\mu$ is ergodic,
 \item $\lambda^{c}(\mu,f)>0$, and
 \item  
$h_\mu(f)>\log \lambda_{u}.$

\end{enumerate}

\subsection{Local unstable manifold for $C^1$ diffeomorphisms}

Depending on the center exponent, Pesin unstable  manifolds may be different from strong unstable manifolds. We will denote the Pesin unstable manifolds by $W^u_{loc}(x)$ or $W^u(x)$, depending on whether they are local or global manifolds, respectively. We denote the strong unstable manifolds by $\cF^{u}(x)$.  We call the intrinsic topology of $\cF^{u}(x)$ the topology given by the restriction of the Riemannian metric of the ambient manifold to $\cF^{u}(x)$.

By (2) and a $C^1$ version of the Pesin theory  \cite[Theorem 3.11]{ABC}, we have that:
\begin{lemma}\label{l.stablemanifold}
For $\mu$ almost every $x$, there is an open set under the  intrinsic topology, $W_{loc}^u(x)\subset \cF^{cu}(x)$ containing $x$, such that
for every $y\in W_{loc}^u(x)$, $$\lim_{n\to \infty} d(f^{-n}(y),f^{-n}(x))=0.$$
\end{lemma}

\subsection{Global unstable manifold}
Recall that for any point $x\in M$, its global Pesin unstable manifold is defined by
$$W^u(x)=\{y;\lim_{n\to \infty} d(f^{-n}(y),f^{-n}(x))=0\}.$$
Remember that $\cF^{cu}$ is the invariant foliation tangent to $E^{cu}$ (see Subsection \ref{subsection.dynamical.coherence}). 
Our main result of this section is the following:

\begin{proposition}\label{p.sizeofstablemanifold} Let $f\in\cD^{1}(A)$ and $\mu$ be an invariant measure satisfying (1) to (3). Then for $\mu$-almost every point $x$,  $\cF^{cu}(x)=W^u(x)$. 
\end{proposition}

\begin{proof} 

 First we will see that it is enough to show that $ \cF^{cu}(x) \subset W^u(x)$ for $\mu$-almost every $x$. In general, we can get the other inclusion from Lemma \ref{l.stablemanifold} but, in this particular case, $\cF^{cu}(x)=W^u(x)$ can be obtained from the fact that $\cF^{cu}(x)$ is a complete immersed submanifold. Indeed, assume that $ \cF^{cu}(x) \subset W^u(x)$. On the one hand, since $\cF^{cu}(x)$ is a surface without boundary, it is an open subset of $W^u(x)$ with the topology induced by the restriction of the ambient Riemannian metric (observe that both manifolds have the same dimension for $\mu$-almost every point $x$). On the other hand, since $\cF^{cu}(x)$ is complete, its boundary as a subset of $W^u(x)$ is empty. The fact that $W^u(x)$ is connected implies $\cF^{cu}(x)=W^u(x)$.
 
 Now we proceed to the prove that   $ \cF^{cu}(x) \subset W^u(x)$ for $\mu$-almost every $x$.  

By Corollary~\ref{c.csproduct}, for any point $x\in \TT^3$,
$$\cF^{cu}_f(x)=\cup_{y\in \cF^c_f(x)} \cF^u_f(y).$$
Hence we only need to show that for $\mu$-almost every
point $x$, $\cF^c_f(x)$ is contained in the unstable manifold of $x$.

The center foliation of the linear Anosov diffeomorphism $A$ is orientable. Also, for every point $x\in \TT^3$ the pre-image of $x$ under the semi-conjugacy $\phi$ is either a point or a connected
center segment of $f$, see (b) of Proposition~\ref{p.coherent}. Therefore the orientation of the center foliation of $A$ induces an orientation on the center foliation of $f$. Once this orientation is fixed, we choose \emph{left} and \emph{right} sides with respect to it.
We denote by $\cF^{c,i}_f(x)$ ($i=right,left$) the points of $\cF^c_f(x)$ which are located on the right and left of
$x$ respectively. Next, we will show that for $\mu$-almost every point $x$, $\cF^{c,right}_f(x)$
belongs to the unstable manifold of $x$. The proof for the left side is similar.

Recall that $\phi$ is the semiconjugacy discussed in Section~\ref{subsection.dynamical.coherence}. By Proposition~\ref{p.isomophic} and Lemma~\ref{l.stablemanifold}, we may take a full $\mu$-measure
subset $\Lambda$ such that for every $x\in \Lambda$:
\begin{itemize}
\item $\phi^{-1}(\phi(x))=x$;

\item $W^u(x)$ contains an open neighborhood of $x$ inside $\cF^c_f(x)$.
\end{itemize}

Now we take a compact subset $\Lambda_0\subset \Lambda$ and $r_0>0$ such that for any $x\in \Lambda_0$,
$\cF^c_{f,r_0}(x)\subset W^u(x)$, where $\cF^c_{f,r_0}(x)$ denotes the ball inside the center leaf $\cF^c_f(x)$
with center $x$ and radius $r_0$. We may take $r_0$ small enough such that  $\mu(\Lambda_0)>0$. Indeed, if $\Lambda_{r}$ is the set of $x\in\Lambda$ such that 
$\cF^c_{f,r}(x)\subset W^u(x)$, then we have $\bigcup_{r}\Lambda_{r}=\Lambda$, so there is $r_{0}>0$ such that $\mu(\Lambda_{r_{0}})>0$. There is a compact set $\Lambda_{0}\subset\Lambda_{r_{0}}$ whose measure approximates $\mu(\Lambda_{r_{0}})$ as much as we wish.
We may further assume the set $\Lambda_0$ is contained in a compact center
foliation box $B$.

Next, we will prove that there is a positive measure subset of $\Lambda_0$, on which $\phi(\cF^{c,right}_{f,r_0}(x))$ has uniform size in $W^u_A(\phi(x))$.

For this purpose, we write $\Lambda_n\subset \Lambda_0$ the set of points such that for any $x\in \Lambda_n$,
such that $\phi(\cF^{c,right}_{f,r_0}(x))$ contains a segment of $\cF^c_{A}(x)$ with length strictly larger than $\frac{1}{n}$.
By the continuity of the center foliation, $\Lambda_n$ is (relatively) open inside $\Lambda_0$, hence measurable.
By the definition, $\Lambda_1\subset \Lambda_2\subset \cdots$, and then the complement of  $\Lambda_\infty=\cup_{n> 0} \Lambda_n$ is a compact set.
\begin{lemma}\label{l.zeroinfty}
$\mu(\Lambda_0\setminus \Lambda_\infty)=0$.
\end{lemma}
\begin{proof}
In order to prove the lemma, it suffices to show that
$$(\phi)_*\mu (\phi(\Lambda_0\setminus \Lambda_\infty))=0.$$
This is because $\phi$ is an isomorphism between $\mu$ and $\phi_{*}\mu$.\par 	
For each point $y\in \Lambda_0\setminus \Lambda_\infty$, denote by $l_{right}(y)\subset \cF^{c,right}_f(y)$ the segment
with length $r_0$ and having $y$ as an endpoint. Then by the choice of $\Lambda_0$,
$l_{right}(y)\subset W^u(y)$. Moreover, since $y\in\Lambda_0\setminus \Lambda_\infty$, the image of $l_{right}(y)$ under the semi-conjugacy $\phi$ must be a single point, which is $\phi(y)$.

Now we claim that for
any two different points $y_1,y_2\in \Lambda_0\setminus \Lambda_\infty$, $l_{right}(y_1)$ and $l_{right}(y_2)$
are disjoint. This is because, if $l_{right}(y_1)\cap l_{right}(y_2)\neq \emptyset$, 
$$\phi(l_{right}(y_1))=\phi(y_1), \mbox{ and }\phi(l_{right}(y_2))=\phi(y_2)$$
must have non-trivial intersection,
which implies that $\phi(y_1)=\phi(y_2)$.  Since 
$\phi\mid_\Lambda$ is bijective, we arrived to a contradiction.

Recall that any one-dimensional segment contains at most countable many disjoint non-trivial intervals. By the claim above,
we conclude that in the center foliation box $B$, the intersection of $\Lambda_0\setminus \Lambda_\infty$ with each center leaf is at most countable. Since the semi-conjugacy maps every center leaf of $f$ to a center leaf of
$A$ (Proposition~\ref{p.coherent}), $\phi(\Lambda_0\setminus \Lambda_\infty)$ intersects every center leaf of $A$ at countably many points.

Note that $\phi_*$ preserves metric entropy (Corollary~\ref{c.preserceentropy}). By the hypothesis of Proposition~\ref{p.sizeofstablemanifold},
$$h_{\phi_*\mu}(A)>\log\lambda_{u}.$$
By Proposition~\ref{p.centerlinear}, the partial entropy along the center foliation $\cF_A^{c}$ of $A$ is positive. 
Then it follows from Proposition~\ref{p.noatom} that the disintegration of $(\phi)_* \mu$ along the center leaf is continuous (in the sense that it contains no atoms), which implies that $(\phi)_*\mu (\phi(\Lambda_0\setminus \Lambda_\infty))=0$, as we claimed.
The proof of this lemma is complete.
\end{proof}

Let us continue the proof of Proposition~\ref{p.sizeofstablemanifold}. Since $\Lambda_\infty$ has full measure in $\Lambda_0$, we can take $m$ sufficiently large, such that $\mu(\Lambda_{m})>0$. By the definition of $\Lambda_{m}$, we have that for $y\in\Lambda_{m}$, the image of $l_{right}(y)$ under $\phi$ has size larger than $1/m$. We claim that the whole right branch of the center leaf, $\cF^c_{right}(x)$,  is contained in $W^u(x)$ for $\mu$-almost every $x\in \Lambda_m$. Although so far our argument has been local, no more than this is required. Indeed, the set of points $x$ for which $\cF^c_{right}(x)\subset W^u(x)$ is $f$-invariant, and $\Lambda_m$, a positive $\mu$-measure set, is contained in it. Ergodicity of $\mu$ implies it has full $\mu$-measure. 

The proof of the claim uses the uniform expansion of $W^u_A$. By Poincar\'e recurrence theorem, for $\mu$ almost every point $x\in \Lambda_{m}$, there is a sequence
of integers $0<n_1<n_2<\cdots$ such that $f^{-n_i}(x)\in \Lambda_{m}$ for any $i\in \mathbb{N}$.
Because $l_{right}(f^{-n_i})(x)\subset W^u(f^{-n_i}(x))$, $f^{n_i}(l_{right}(f^{-n_i})(x))\subset W^u(x)$.
By the semi-conjugacy, $$\phi(f^{n_i}(l_{right}(f^{-n_i})(x)))=A^{n_i}(\phi(l_{right}(f^{-n_i})(x)))$$
satisfies
\begin{align*}
\length(A^{n_i}(\phi(l_{right}(f^{-n_i})(x))))>
\lambda_{c}^{n_i} \length(\phi(l_{right}(f^{-n_i})(x))))\geq \lambda_{c}^{n_i} \frac{1}{m},
\end{align*}
which is unbounded in $i$. Hence $$\phi(\bigcup_{i}f^{n_i}(l_{right}(f^{-n_i})(x)))=\bigcup_{i}\phi(f^{n_i}(l_{right}(f^{-n_i})(x)))=\cF^c_{A,right}(\phi(x)).$$
This shows that $$\bigcup_{i}f^{n_i}(l_{right}(f^{-n_i})(x))=
\cF^c_{f,right}(x)\subset W^u(x).$$ The proof of Proposition~\ref{p.sizeofstablemanifold} is complete.
\end{proof}

\begin{remark}
	The assumption (3) in Proposition~\ref{p.sizeofstablemanifold} is likely a sharp condition. See the discussion in~\cite{VY2}.
\end{remark}

\section{Robustly minimal stable foliation\label{s.sminimal}}

Finally we are ready to prove Theorem~\ref{m.boundedpotention}. Due to the non-uniform expansion on $\cF^{cu}$, we need the notion of hyperbolic times.  These are the times when sufficient hyperbolicity is achieved along a given orbit. We will prove that, for any given open set $U$, there is always some point $x\in U$ whose forward iteration (up to the hyperbolic times) has large unstable manifold. Then Propositions~\ref{p.sizeofstablemanifold} and \ref{p.product} show that under further iteration this unstable manifold will become an $s$-section, intersecting every stable leaf.  Ergodic measures  in $\G(g)$ satisfy the hyphotesis of Proposition~\ref{p.sizeofstablemanifold}   because they have large entropy by Lemma \ref{l.robustuniform}. As a result,  it will be enough to show that when such hyperbolicity is achieved, the point itself must be close to the support of some measure in $\G(g)$.

Let $\cU$ be the $C^1$ neighborhood of $f$ and $a>0$ be the constant provided by Lemma~\ref{l.robustuniform}. Let  $g\in \cU$
be a $C^1$ diffeomorphism. Given $U$  any open set of the ambient manifold $\TT^3$,
it suffices  to show that every stable leaf has nonempty intersection with $U$.

By Proposition~\ref{p.meGspace}, there is a full volume subset $\Gamma\subset U$ such that for any $x\in \Gamma$,
any limit of the sequence of $\frac{1}{n}\sum_{i=0}^{n-1}\delta_{g^i(x)}$ belongs to $\G(g)$. Fix such an $x\in \Gamma$, then
$$\limsup\frac{1}{n}\sum_{i=1}^{n} \log \|Tg^{-1}\mid_{E_g^{c}(g^{i}(x))}\|<-a. $$ 
Because all the subbundles $E^i$ ($i=s,c,u$) are one-dimensional, after changing the metric, we may assume that they are orthogonal. This means that, after changing the constant $a$, we have:
\begin{equation}\label{eq.eventurallyhyperbolic}
\limsup\frac{1}{n}\sum_{i=1}^{n} \log \|Tg^{-1}\mid_{E_g^{cu}(g^{i}(x))}\|<-a.
\end{equation}

\begin{definition}\label{df.hyperbolictime}
For $b>0$, we say that $n$ is a $b$-\emph{hyperbolic time} for a point $x$, if
$$\frac{1}{k}\sum_{j=n-k+1}^n\log\|Tg^{-1}\mid_{E^{cu}(h^j(x))}\|\leq -b, \text{ for any } 0<k \leq n.$$
\end{definition}

By the Pliss Lemma (see \cite[Lemma 11.5]{BDVnonuni})), one can show that the set of $b$-hyperbolic times have positive density:

\begin{lemma}\label{l.largehyperbolictime}
There is $\rho=\rho(a,\cU)>0$ such that for any $g\in \cU$ and $x$ satisfying \eqref{eq.eventurallyhyperbolic},
there are integers $0<n_1<n_2<\cdots$ which are $\frac{a}{2}$-hyperbolic times for $x$. Moreover,
$$\liminf \frac{\#\{i:n_i\leq n\}}{n}\geq \rho.$$
\end{lemma}

Write $\nu_{n_m}=\frac{1}{n_m}\sum_{i=0}^{n_m-1} \delta_{g^i(x)}$. We may assume that
$$\lim_m \nu_{n_m}=\lim_m \frac{1}{n_m}\sum_{i=0}^{n_m-1}\delta_{g^i(x)}=\nu.$$
By Proposition~\ref{p.physical}, we have $\nu\in \G(g)$.

For each $m$, define the set
$$\Gamma_m=\{g^i(x);\; n_m \rho/2<i<n_m \text{ is a hyperbolic time of } x \}.$$

By Lemma~\ref{l.largehyperbolictime}, $\liminf_m \nu_{n_m}(\Gamma_m)\geq\rho/2$. Take $\Gamma_0$ any Hausdorff limit
of the sequence $\Gamma_m$. For simplicity, passing to a subsequence if necessary, we may assume that $\lim \Gamma_m=\Gamma_0$.
Then $\nu(\Gamma_0)\geq \rho/2$. The measure $\nu$ could be non-ergodic.  However, we claim that there is an ergodic component $\nu'$ of $\nu$ such that $\nu^\prime(\Gamma_0)\geq \rho/2$. Otherwise  we would have that $\nu(\Gamma_0)< \rho/2$, a contradiction. So we take an ergodic component $\nu^\prime$ of $\nu$ such that $\nu^\prime(\Gamma_0)\geq \rho/2$. 
Since $\nu\in \G(g)$, by Proposition~\ref{p.meGspace} we can choose $\nu'$ in such a way that  $\nu^\prime\in G(g)$.

By Lemma~\ref{l.robustuniform}, $h_{\nu^\prime}(g)>\log \lambda_{u}$. Then Proposition~\ref{p.sizeofstablemanifold} shows that
for $\nu^\prime$-almost every point $y\in \Gamma_0$, $W_g^u(y)=\cF_g^{cu}(y)$. By Poincar\'e recurrence theorem,
we may also assume that the negative orbit of $y$ visits $\Gamma_0$ infinitely many times.

By Remark~\ref{r.uniformintersection}, there is $R=R_{y}$ such that $\cF^{cu}_{g,R}(y)$ intersects each stable leaf of $g$.
Because $W^u_g(y)=\cF^{cu}_g(y)$, there is a sufficiently large $s$ such that $g^{-s}(y)\in \Gamma_0$, and $g^{-s}(W^u_{g,R}(y))$
has arbitrarily small size, but also intersects every stable leaf of $g$.

On the other  hand, since $g^{-s}(y)\in \Gamma_0=\lim \Gamma_m$, there is a  subsequence $n_{i_m}$ of hyperbolic times of $x$, such that  $\Gamma_m \ni x_{n_{i_m}}:=g^{n_{i_m}}(x)\to g^{-s}(y)$ as $i\to\infty$.

Let $D$ be any two-dimensional $\C^1$ disk. We use $d_D(\cdot,\cdot)$ to denote the distance induced on $D$ by the restriction  of the Riemammian metric of $\TT^3$. We call {\em $1/2$ center-unstable cone}  the set of vector fields $v$ such that $\angle (v, E^{cu})<\frac{1}{2}$.
 Since the splitting $E^s\oplus E^{cu}$ is dominated, the vectors in the sub-bundle $E^{cu}$ are expanded exponentially faster than  vectors in the sub-bundle $E^s$. It then follows that:
\begin{lemma}\label{l.cucone}
Suppose $g\in\cU$ and $D$ is a disk tangent to the $1/2$ center-unstable cone, then $g(D)$ is also tangent to $1/2$ center-unstable cone.
\end{lemma}

More importantly, one sees sufficient backward contraction on a large size sub-disk of $g^k(D)$ when  $k$ is a hyperbolic time:

\begin{lemma}\cite[Lemma 2.7]{ABV}\label{l.hyperbolictime}
There is $\delta_1>0$ depending on $\cU$ and $a$ such that, for any diffeomorphism $g\in \cU$, given any $\C^1$ disk $D$ tangent to the $1/2$ center-unstable cone field, $x\in D$ and $n\geq 1$ an $a/2$-hyperbolic time for $x$, we have
$$d_{g^{n-k}(D)}(g^{n-k}(y),g^{n-k}(x))\leq e^{-ka/2} d_{g^n(D)}(g^n(x),g^n(y)),$$
for any point $y\in D$ with $d_{g^n(D)}(g^n(x),g^n(y))\leq \delta_1$.
\end{lemma}

Take $D\subset U$ any two-dimensional disk which is tangent to the $1/2$ center-unstable cone and contains $x\in\Gamma$ as an interior point. By Lemma~\ref{l.cucone} and \ref{l.hyperbolictime}, for any sufficiently large $n_i$ which is $a/2$-hyperbolic time
for $x$, $D_{n_i}=g^{n_i}(D)$ is tangent to the $1/2$
center-unstable cone, and contains a sub-disk with center $x_{n_i}$ and radius $\delta_1$ with respect to the distance $d_{g^{n_i}(D)}$.

Since $\delta_1>0$ only depends on $a>0$ and $\mathcal U$, one can take a large enough $s$  such that $g^{-s}(W^u_{g,R}(y))$ is much smaller than $\delta_1>0$.
Then, because $x_{n_{i_m}}$ can be taken arbitrarily close to $g^{-s}(y)$, all the stable manifold of points of  $g^{-s}(W^u_{g,R}(y))$ intersect  $D_{n_{i_m}}$. Indeed, observe that the disks $D_{n_{i_m}}$ are uniformly transverse to $\cF^s_g$ and contain disks with  uniform  radius centered at $x_{n_{i_m}}$, while the diameter of $g^{-s}(W^u_{g,R}(y))$ can be taken arbitrary small. Then the continuity of the strong stable foliation implies that  all the stable manifolds of points of  $g^{-s}(W^u_{g,R}(y))$  intersect $D_{n_{i_m}}$ as stated.
Since $g^{-s}(W^u_{g,R}(y))$ intersects every stable leaf of $g$,   $D_{n_{i_m}}$ also does, and  so does $D$. The proof of our main result is complete.

\section{Examples\label{s.examples}}

In this section we will build the new examples mentioned in Theorem \ref{main.example}. For this we will first consider  the sequence of hyperbolic automorphisms $\{A_k\}$ on $\TT^3$ presented in \cite{PT}. In fact we will take their inverses. Each automorphism in this sequence has three distinct positive eigenvalues. The center eigenvalues are greater than one and tend to one as $k\to\infty$. This means that the central expansion is getting arbitrarily weak.  This  will allow us to make a perturbation on a sufficiently long center segment so that the new diffeomorphism is the identity on this segment, and  has center expansion elsewhere.  We have to make the perturbation taking two precautions. On the one hand, we want the perturbed diffeomorphism to be partially hyperbolic. On the other hand, we want the expansion in the strong unstable direction not to be affected in order to  apply  Theorem \ref{m.boundedpotention}. In this way, we will obtain the robust minimality of  the strong stable foliation. 
 
In addition, we need the strong unstable foliation also to be robustly minimal. To achieve this we require that the center segments in which the perturbations occur are long enough to  apply the arguments of \cite{BDU}. This is possible thanks to a density  estimate for long center segments given by Lemma \ref{l.separate}.
 

\subsection{The hyperbolic automorphisms in~\cite{PT} and a change of coordinates}
For each $k\in \mathbb{Z}$, we define the linear automorphism $A_k : \TT^3\longrightarrow \TT^3$
induced by the integer matrix
$$A_k=\left(
  \begin{array}{ccc}
    k-1 & -1 & -1 \\
    1 & 1 & 0 \\
    1 & 0 & 0 \\
  \end{array}
\right).$$
For simplicity of the notation we identify both the automorphism and the linear matrix by $A_k$.

The inverse of $A_k$ is 
$$\left(
  \begin{array}{ccc}
    0 & 0 & 1 \\
    0 & 1 & -1 \\
    -1& -1 & k \\
  \end{array}
\right).$$
which is the matrix discussed in Section 4 of \cite{PT}.
The characteristic polynomial of $A_k$ is 
$$p_k(x)=x^3-kx^2+(k+1)x-1=0.$$

Moreover,
\begin{lemma}\cite[Section 4]{PT}  \label{p.eigenvalueandvector}
For all $k\geq 5$, $A_k$ has real eigenvalues $$0<\lambda^s_k<\frac{1}{k}<1<\lambda^c_k<\lambda^u_k$$
which satisfy $\lambda^s_k\searrow 0$, $\lambda^c_k\searrow 1$ and $\lambda^u_k\nearrow \infty$ as $k\to \infty$. The eigenvectors
are $$v^\sigma_k=\left(1,\frac{1}{\lambda^\sigma_k-1},\frac{1}{\lambda^\sigma_k}\right), \sigma=s,c,u.$$
\end{lemma}
\begin{remark}\label{rk.eigendirection}
It is easy to see that for $k$ sufficiently large, one has

\begin{itemize}
\item $p_k(0)=-1<0,$
\item $p_k(\frac{1}{k})=\frac{1}{k^3}>0,$
\item $p_k(1)=1>0,$
\item $p_k(1+\frac{1}{k})=\frac{1}{k^3}+\frac{3}{k^2}+\frac{3}{k}>0,$
\item $p_k(2)=-2k+9<0,$
\item $p_k(\frac{k}{2})=-\frac{k^3}{8}+\frac{k^2}{2}+\frac{k}{2}-1<0,$ 
\item  $p_k(k)=k^2+k-1>0.$

\end{itemize}
 Therefore we can locate the eigenvalues of $A_k$ on the real line as follows:
\begin{equation}\label{eq.eigenv}
0<\lambda^s_k<\frac{1}{k}<1+\frac{1}{k}<\lambda^c_k<2<\frac{k}{2}<\lambda^u_k<k.
\end{equation}

Let $e^\sigma_k=v^\sigma_k/|v^\sigma_k|$ and denote by $E^\sigma_k = \mbox{span}\{e^\sigma_k\}$ the eigenspaces,  where $\sigma=s,c,u$.
Then we have, as $k\to \infty$, the following limits:
\begin{equation}\label{base.k}
e^s_k\to \left(
    \begin{array}{c}
      0 \\
      0 \\
      1 \\
    \end{array}
  \right),
\;\;
e^c_k\to \left(
    \begin{array}{c}
      0 \\
      1 \\
      0 \\
    \end{array}
  \right),
  \;\;
  e^u_k\to \left(
    \begin{array}{c}
      1 \\
      0 \\
      0 \\
    \end{array}
  \right)
.\end{equation}
\end{remark}

To simplify our computations  we consider the representation  of the matrix $A_k$ under the basis $\cB_k = \{e^u_k,e^c_k, e^s_k\}$, which is  
$$[A_k]_{\cB_k} = \left(
  \begin{array}{ccc}
    \lambda^u_k & 0 & 0 \\
    0 & \lambda^c_k & 0 \\
    0 & 0 & \lambda^s_k \\
  \end{array}
\right).$$
Henceforth, unless otherwise specified, we will use $\cB_k$ as the coordinate basis for $\RR^3$.

\subsection{A perturbation of the identity inside a cylinder}
In this subsection we will construct a perturbation $\Psi_{a,b,c,d}$ inside a solid cylinder $\mathcal{T}_{a,d}$ with length $a>1$ and radius $d>0$:
\begin{equation}\label{e.T}
\mathcal{T}_{a,d}=\{[x,y,z]_{\cB_k}: y\in[-a,a], x^2+z^2 \le d^2\}\subset \RR^3.
\end{equation}
We will also take $c\in (-1,0)$ and $b\in(0,a)$, whose purpose will become apparent soon. We remark that all constants $a,b,c$ and $d$, as well as the basis $\cB_k$, depend on $k$; however, since our construction is carried out for all $k$ sufficiently large and to improve the readability, we will suppress the dependence  on $k$ for now.

We begin with a $C^\infty$ bump function $\rho_d: [0,d]\to \mathbb{R}$ such that \label{bump.function}
\begin{itemize}
\item[(a1)] $\rho_d(0)=1$ and $\rho_d^\prime(0)=0$;
\item[(a2)] $\rho_d(d)=0$ and $\rho_d^\prime(d)=0$;
\item[(a3)] $\rho_d(x)$ is decreasing;
\item[(a4)] $-\frac{4}{d}<\rho_d^\prime < 0$ for $0<x\leq d$.
\end{itemize}
A typical choice is 
$$
\rho_1(x) = \begin{cases}
\exp(1-\frac{1}{1-x^2}), &x\in[0,1)\\
0,&x=1
\end{cases}
$$
and rescaled to $[0,d]$ by $\rho_d(x) = \rho_1(x/d)$.

We will use a bump function $\phi_{a,b,c}$, obtained in the lemma below, to construct the perturbation $\Psi_{a,b,c,d}$ along the center of the cylinder $\mathcal T_{a,b}$:
\begin{lemma}\label{l.1dmap}
Given $a>1$, $b\in(0,a)$ and $c\in(-1,0)$ there exists a $C^\infty$ function $\phi_{a,b,c}: [-a,a]\to \mathbb{R}$ such that:
\begin{itemize}
\item[(b1)] $\phi_{a,b,c}(x)=cx$ for $x\in [-b,b]$;
\item[(b2)]  $\phi_{a,b,c}(a)=\phi_{a,b,c}(-a)=0$, and $\phi_{a,b,c}^\prime(a)=\phi_{a,b,c}^\prime(-a)=0$;
\item[(b3)] $c\leq  \phi_{a,b,c}^\prime(x) < 2\frac{b}{a-b}|c|$ for all $x\in [-a,a]$, $\phi_{a,b,c}^\prime(x)=c$ if and only if $x\in [-b,b]$;
\item[(b4)] $|\phi_{a,b,c}(x)|<2b|c|$.
\end{itemize}
\end{lemma}
\begin{proof} 
We will only construct $\phi_{a,b,c}$ on $[0,a]$ and extend it to $[-a,a]$ as an odd function.

Let $\vep>0$ be sufficiently small so that $b+\vep<a-\vep$. We begin with the piecewise linear function:
$$
L_{a,b,c;\vep}(x) = \begin{cases}
cx, &x\in(-\infty,b+\vep);\\
0, &x\in [a-\vep,\infty);\\
\frac{-(b+\vep)c}{a-b-2\vep}\big(x-(a-\vep)\big), &\mbox{otherwise}.
\end{cases}
$$
The function $L_{a,b,c;\vep}(x)$ is continuous with 
$$
\frac{-(b+\vep)c}{a-b-2\vep} <2\frac{b}{a-b}|c|
$$
if $\vep>0$ is small enough. See Figure~\ref{f.L}.

\begin{figure}
	\centering
	\def\svgwidth{\columnwidth}
	\includegraphics[scale = 0.8]{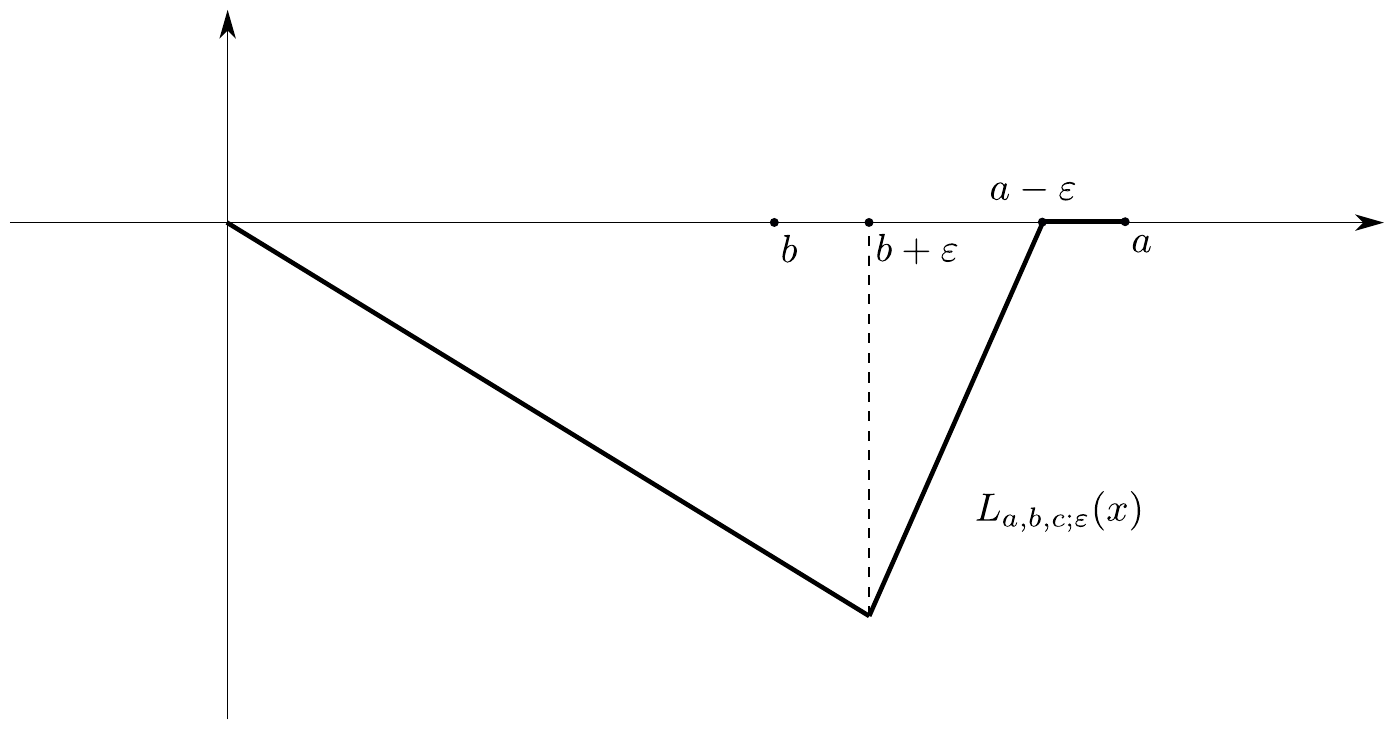}
	\caption{The piecewise linear function $L_{a,b,c;\vep}$ restricted to $[0,a]$.}
	\label{f.L}
\end{figure}

To obtain $\phi_{a,b,c}$ we fix $\vep>0$ and define
\begin{equation*}\label{e.phi}
\phi_{a,b,c} = L_{a,b,c;\vep}* \varphi_\vep
\end{equation*}
where $\varphi_\vep(x)$ is a symmetric mollifier whose support is $(-\vep,\vep)$ and $*$ denotes the convolution. It is straightforward to check that $\phi_{a,b,c}\mid_{[0,a]}$ satisfies all the required properties.
\end{proof}

We define a $C^\infty$ map $\Psi_{a,b,c,d}: \mathcal{T}_{a,d}\to \mathcal{T}_{a,d}$ by:
\begin{equation}\label{e.psi}
\Psi_{a,b,c,d}([x,y,z]_\cB)=[x,y+\rho_d(r)\phi_{a,b,c}(y),z]_\cB,
\end{equation}
where $r=\sqrt{x^2+z^2}$ and $\mathcal T_{a,d}$ is the cylinder defined in ~\eqref{e.T}.
Its derivative (under the coordinate basis $\cB$) is given by:
\begin{equation}\label{e.Tpsi}
[T\Psi_{a,b,c,d}]_\cB=\left(
    \begin{array}{ccc}
     1 & 0 & 0 \\
       \phi_{a,b,c}(y)\rho_d^\prime(r)\frac{x}{r} &  1+\rho_d(r)\phi_{a,b,c}^\prime(y) & \phi_{a,b,c}(y)\rho_d^\prime(r)\frac{z}{r} \\

      0 & 0 & 1 \\
    \end{array}
  \right)
\end{equation}
$$
=\left(
    \begin{array}{ccc}
     1 & 0 & 0 \\
       C_1 &  C_2 & C_3 \\

      0 & 0 & 1 \\
    \end{array}
  \right),
$$
where $C_1,C_2$ and $C_3$ are smooth functions of $[x,y,z]_\cB$.

The lemma below follows now immediately from ~\eqref{e.psi} and~\eqref{e.Tpsi}.
\begin{lemma}\label{l.tube}
For any $a>1$, $b\in(0,a)$, $c\in(-1,0)$ and $d>0$, the following properties hold:
\begin{itemize}
\item $\Psi_{a,b,c,d}:\mathcal{T}_{a,d}\to \mathcal{T}_{a,d}$ is a $C^\infty$ diffeomorphism isotopic to identity  that  preserves any line parallel to the $y$-axis;
\item $\Psi_{a,b,c,d}\mid_{\partial \mathcal{T}}=id$ and $T\Psi_{a,b,c,d}\mid_{\partial \mathcal{T}}=Id$;
\item $T\Psi_{a,b,c,d}$ preserves any 2-dimensional linear subspace which is parallel to the $y$-axis;
\item $\det T\Psi_{a,b,c,d}\ge 1+c>0$. The equality holds if and only if  $x=z=0$ and $y\in[-b,b]$.
\item 
\begin{equation}\label{eq.boundinsidetube1}
1+c\le C_2\le  1+2\frac{b}{a-b}|c|\quad\mbox{ and }\quad
 \max\left\{|C_1|,|C_3|\right\} \leq \frac{8b}{d}|c|.
\end{equation}

\end{itemize}
\end{lemma}

\subsection{The perturbed map $f_k$}\label{ss.7.3}
We are now ready to construct the perturbed map $f_k$. For each $k\in\ZZ$ we take:
\begin{itemize}
\item $a_k>1$;
\item $\theta_k\in(0,(\lambda^c_k)^{-1}]$ such that $b_k:=\theta_ka_k>1$;
\item $c_k = (\lambda^c_k)^{-2}-1\in(-1,0)$;
\item $d_k>0$ small enough such that the projection $\pi:\RR^3\to \TT^3$, restricted on $\mathcal T_{a_k,d_k}$, is injective.
\end{itemize}
The properties above and Lemma~\ref{l.tube} show that $\psi_k:=\Psi_{a_k,b_k,c_k,d_k}$ can be viewed as a $C^\infty$ diffeomorphisms on $\TT^3$ with $\psi_k = id$ outside $\pi\left(\mathcal T_{a_k,d_k}\right)$. We also define the following center segments:
\begin{equation}\label{e.Ik}
J_k = \bigcup_{t\in[-a_k,a_k]} te^c_k,\,\,  \mbox{ and } I_k = \bigcup_{t\in[-b_k,b_k]} te^c_k.
\end{equation}
It follows that $A_k(I_k)\subset J_k\subset \mathcal T_{a_k,d_k}$. From Lemma~\ref{l.tube} we see that:

\begin{enumerate}
\item $\psi_k$ fixes the center leaves of $A_k$;
\item $T\psi_k$ preserves the bundle $E_{A_k}^{cu}$;
\item $T\psi_k\mid_{E^c_k} \ge 1+c_k = (\lambda^c_k)^{-2}$. The equality holds only on $I_k$.
\end{enumerate}

Finally we define 
\begin{equation}\label{e.fk}
f_k = A_k\circ \psi_k\circ A_k.
\end{equation}
We state the following proposition which immediately implies Theorem~\ref{main.example}.

\begin{proposition}\label{proposition}
There is $k_0>0$ such that for every $k\geq k_0$, the  diffeomorphism $f_k$ defined by~\eqref{e.fk} is  partially hyperbolic and is contained in the isotopy class of the linear Anosov diffeomorphism $B_k = (A_k)^2$. Furthermore:
\begin{itemize}
\item[(a)] $\cB(f_k)=\{x: | \det(Tf_k\mid_{E_{f_k}^{cu}(x)}) | \leq \lambda_{u} = (\lambda^u_k)^2\}$ coincides with $I_k$, the center segment defined by~\eqref{e.Ik}; consequently, $f_k$ has robustly minimal stable foliation;
\item[(b)] $\|Tf_k|_{E^c_{f_k}(x)}\|=1$ for every $x\in I_k$, and $\|Tf_k|_{E^c_{f_k}(x)}\|>1$ otherwise;

\item[(c)] if in addition there exists  $C>0$ such that $|c_k|/d_k = (1-(\lambda^c_k)^{-2})/d_k<C$ for all $k$ large enough, then for $k$ sufficiently large, there exists an open set $\cU_k$, containing $f_k$ in its closure, such that every $g\in\cU_k$ has a minimal unstable foliation.
\end{itemize}
Moreover, it is possible to choose $a_k$, $b_k$ and $d_k$ such that  $\length(I_k)\to\infty$ as $k\to\infty$, and $\lim_H I_k = \TT^3$.
\end{proposition}

The fact that $\cB(f_k)=\{x: | \det(Tf_k\mid_{E_{f_k}^{cu}(x)}) | \leq \lambda_{u} = (\lambda^u_k)^2\}$ coincides with $I_k$ and item (b) follow directly from the construction of $f_{k}$ (in particular from the choice of $c_k$). In Section \ref{ss.PH} we establish the partial hyperbolicity of $f_{k}$. After this is done, the robust minimality of the stable foliation follows directly from Theorem~\ref{m.boundedpotention}, since the leaf volume of $\cB(f_{k})=I_{k}$ inside any strong unstable leaf is zero. The rest of this section is devoted to the proof of this proposition: The partially hyperbolicity of $f_{k}$ is proven in Section~\ref{ss.PH}. The robust minimality of the unstable foliation is proven in Section~\ref{ss.minimal}. And the length and denseness of $I_k$ are proven in Section~\ref{ss.density}.

\subsection{Partial hyperbolicity\label{ss.PH}}
In this subsection we show that for $k$ sufficiently large, $f_k$ is  partially hyperbolic. 

We use the following classical criterion, whose proof is standard and can be found in the Appendix.
\begin{lemma}\label{l.dominated}
Suppose $f\in \Diff(\TT^3)$ admits two invariant two-dimensional subbundles $E$ and $G$ which are transverse to each other
at any point. Denote by $F=E\cap G$  which is a one-dimensional invariant subbundle. Suppose $f$ admits closed cones $\cC^E\subset E$ and $\cC^G\subset G$
both transverse to $F$, such such that $Tf(\cC^G)\subset \Int(\cC^G)$ and $Tf^{-1}(\cC^E)\subset \Int(\cC^E)$.
Then $f$ admits a dominated splitting $T\TT^3=E^\prime\oplus F \oplus G^\prime$, where $E^\prime=\cap_{n\ge 0} Tf^{-n}(\cC^E)$
and $G^\prime=\cap_{n\ge 0} Tf^n(\cC^G)$. Moreover, $f$ is partially hyperbolic if for any $x\in \TT^3:$
\begin{equation}\label{eq.partialhyperbolic}
\begin{split}
&|\det Tf\mid_{E(x)}|<|\det Tf\mid_{F(x)}|=\|Tf\mid_{F(x)}\|\\ &\text{  and  } \\ &|\det Tf\mid_{G(x)}|>|\det Tf\mid_{F(x)}|=\|Tf\mid_{F(x)}\|.
\end{split}
\end{equation}
\end{lemma}

We will apply this criterion to establish the partial hyperbolicity of $f_{k}$ for sufficiently large $k$, where $f_{k}$ is defined in ~\eqref{e.fk}. Recall that under the coordinate basis $\cB_k$, the derivative $T\psi_k$ has the form
$$
[T\psi_k]_{\cB_k}=
    \begin{pmatrix}
     1 & 0 & 0 \\
       C^{k}_1 &  C^{k}_2 & C^{k}_3 \\

      0 & 0 & 1
    \end{pmatrix}
$$
where $C^{k}_1, C^{k}_2$ and $C^{k}_3$ are the smooth functions of $[x,y,z]_{\cB_k}$ defined in ~\eqref{e.Tpsi}.
This means that 
$$
[Tf_k]_{\cB_k}=
    \begin{pmatrix}
      (\lambda^u_k)^2 & 0 & 0 \\
     \lambda^c_k\lambda^u_kC^{k}_1 &  (\lambda^c_k)^2 C^{k}_2 & \lambda^c_k\lambda^s_kC^{k}_3 \\
      0 & 0 & (\lambda^s_k)^2 \\
    \end{pmatrix}.
$$

\noindent We will let $E$ and $G$ be the $XY$ and $YZ$-plane, respectively in the $\cB_{k}$ coordinates. Then $E$ and $G$ are both invariant under $Tf_k$ and $F=E\cap G$ is the $y$-axis.

Recall that $b_k = \theta_k a_k$ with $\theta_k\in(0,(\lambda^c_k)^{-1})$, and $c_k = (\lambda^c_k)^{-2}-1$. Then we obtain the following bounds of $C^{k}_2$ from~\eqref{eq.boundinsidetube1} 
\begin{align*}
(\lambda^c_k)^{-2}\le C^{k}_2&\le 1+2\frac{\theta_k}{1-\theta_k}\left(1-(\lambda^c_k)^{-2}\right)\\
&\le 1+2\frac{(\lambda^c_k)^{-1}}{1-(\lambda^c_k)^{-1}}\frac{(\lambda^c_k+1)(\lambda^c_k-1)}{(\lambda^c_k)^{2}}\\
&=1+2\frac{\lambda^c_k+1}{(\lambda^c_k)^2}.
\end{align*}
Noting that $\lambda_k^c\searrow 1$ as $k\to\infty$, for $k$ large we have
\begin{equation}\label{e.C2}
(\lambda^c_k)^{-2}\le C^{k}_2\le 6.
\end{equation}

We first build the  cone family transverse to $F$ inside the fiber bundle $E$. 
Then we only have to consider the restriction of $Tf_k$ to $E$, which has the form:
$$[Tf_k\mid_{E}]_{\cB_k}=\left(
    \begin{array}{cc}
     (\lambda^u_k)^2  &  0\\
     \lambda^c_k\lambda^u_kC^{k}_1 &  (\lambda^c_k)^2C^{k}_2\\
    \end{array}
  \right).$$

For $K^u>0$ which will be specified later, we consider the cone inside $E = \RR^{XY}$:
$$\cC_{K^u}=\left\{\left[
    \begin{array}{c}
      u \\
      v \\
    \end{array}
  \right]_{\cB_k}: u\neq 0, \text{ such that } \frac{|v|}{|u|}\leq K^u\right\}.$$
  \medskip
  
\noindent {\bf Claim.} {\it For $K^u_{k}>0$ sufficiently large, $\cC_{K^u_{k}}$ is invariant under $Tf_k$, that is, $Tf_k(\cC_{K^u_{k}})\subset \Int(\cC_{K^u_{k}}).$ }

\medskip

{\em\noindent Proof of the claim}. For any $\left[
    \begin{array}{c}
      1 \\
      v \\
    \end{array}
  \right]_{\cB_k}\in \cC_{K^u}$, 
  
  $$Tf_k\left[
    \begin{array}{c}
      1 \\
      v \\
    \end{array}
  \right]_{\cB_k}=\left[
    \begin{array}{c}
  (\lambda^u_k)^2  \\
        \lambda^c_k\lambda^u_kC^{k}_1+ (\lambda^c_k)^2C^{k}_2v \\
    \end{array}
  \right]_{\cB_k}.$$
  
  By~\eqref{eq.boundinsidetube1} and~\eqref{e.C2}, we obtain
  \begin{align*}
&\left|\frac{\lambda^c_k\lambda^u_kC^{k}_1+ (\lambda^c_k)^2C^{k}_2v}{(\lambda^u_k)^2}\right|\\
\le &6\left(\frac{\lambda^c_k}{\lambda^u_k}\right)^2 K^u+8\frac{\lambda^c_k}{\lambda^u_k}\frac{b_k}{d_k}|c_k|\\
:= &\, \Theta_k K^u + M_k. 
  \end{align*}
Note that $\Theta_k = 6(\lambda^c_k/\lambda^u_k)^2\to0$ as $k\to\infty$. We see that the choice 
  \begin{equation}\label{e.Ku}
  K^u_{k} := \frac{2}{1-\Theta_k}M_k = \frac{16\frac{\lambda^c_k}{\lambda^u_k}\frac{b_k}{d_k}|c_k|}{1-6\left(\frac{\lambda^c_k}{\lambda^u_k}\right)^2}
  \end{equation}
  gives, for $k$ sufficiently large,
  $$
 \left| \frac{\lambda^c_k\lambda^u_kC^{k}_1+ (\lambda^c_k)^2C^{k}_2v}{(\lambda^u_k)^2}\right|\le \Theta_k K^u_{k} + M_k = \frac{2\Theta_k}{1-\Theta_k}M_k+M_k = \frac{1+\Theta_k}{1-\Theta_k}M_k<K^u_{k}.
  $$
This shows that $Tf_k([1,v]^T_{\cB_k})\in \Int (\cC_{K^u_{k}})$ for any $[1,v]^T_{\cB_k}\in \cC_{K^u_{k}}$, and consequently,
$$Tf(\cC_{K^u_{k}})\subset \Int(\cC_{K^u_{k}}).$$
The proof of the claim is finished.

Now we  build the  cone family transverse to $F$ inside the invariant bundle $G = \RR^{YZ}$. The construction is analogous, with
$$\cC_{K^s}=\left\{\left[
\begin{array}{c}
      u \\
      v \\
    \end{array}
\right]_{\cB_k}\mid v\neq 0, \text{ such that } \frac{|u|}{|v|}\leq K^s\right\}
$$
for some $K^s>0$.
  \medskip

\noindent {\bf Claim.} {\it For $K^s_{k}>0$ sufficiently large, $\cC_{K^s_{k}}$ is $Tf_k^{-1}$-invariant; that is, $Tf^{-1}(\cC_{K^s_{k}})\subset \Int(\cC_{K^s_{k}}).$ }

\medskip

{\em\noindent Proof of the claim}. The proof is similar to the previous case. Direct calculation shows that for any $\left[
    \begin{array}{c}
      u \\
      1 \\
    \end{array}
  \right]_{\cB_k}\in \cC_{K^s}$,

  $$Tf^{-1}_k\left[
    \begin{array}{c}
    u \\
      1 \\
    \end{array}
  \right]_{\cB_k}=\left[
    \begin{array}{c}
    \frac{u}{(\lambda^c_k)^2C^{k}_2}-\frac{C^{k}_3}{\lambda^c_k\lambda^s_kC^{k}_2}  \\
     (\lambda^s_k)^{-2}  \\
    \end{array}
  \right]_{\cB_k}.$$
This leads to 
  \begin{align*}
&\left|(\lambda^s_k)^2\left(\frac{u}{(\lambda^c_k)^2C^{k}_2}-\frac{C^{k}_3}{\lambda^c_k\lambda^s_kC^{k}_2}\right)\right|\\
\le & \left(\frac{\lambda^s_k}{\lambda^c_k}\right)^2(C^{k}_2)^{-1}|u| + \frac{\lambda^s_k}{\lambda^c_k}\frac{C^{k}_3}{C^{k}_2}\\
\le&\left(\lambda^s_k\right)^2 K^s + 8\lambda^s_k\lambda^c_k\frac{b_k}{d_k}|c_k|\\
:= &\, \Theta'_k K^s + M'_k,
  \end{align*}
where we apply again~\eqref{eq.boundinsidetube1} and~\eqref{e.C2} to obtain the last inequality. 
Then the choice of 
  \begin{equation}\label{e.Ks}
  K^s_{k} = \frac{2}{1-\Theta'_k}M'_k =  \frac{16\lambda^s_k\lambda^c_k\frac{b_k}{d_k}|c_k|}{1-(\lambda^s_k)^2}
  \end{equation}
suffices. This concludes the proof of the claim.

Now we are in a position to apply Lemma~\ref{l.dominated}, which shows that $f_k$ admits a dominated splitting $T\TT^3=E'\oplus F\oplus G'.$ To show that $f_k$ is indeed partially hyperbolic, we only need to verify~\eqref{eq.partialhyperbolic}. However, this is easy since 
\begin{align*}
|\det Tf_k\mid_{E}|=&| (\lambda^u_k\lambda^c_k)^2C_2|\\
>&|(\lambda^c_k)^2C_2|\\
=&|\det Tf_k\mid_{F}|.
\end{align*}
The other inequality from~\eqref{eq.partialhyperbolic} is similar and thus omitted. With Lemma~\ref{l.dominated} we conclude that $f_k$ is partially hyperbolic for $k$ sufficiently large.

\begin{remark}
Note that the size of the cones are not uniform in $k$. However, with extra conditions on $b_k$ and $d_k$ it is possible to choose the same cones for all $f_k$ with large enough $k$. See Lemma~\ref{l.separate} in the next subsection.
\end{remark}

\subsection{Robustly minimal unstable foliations}\label{ss.minimal}
By construction we have
$$
\cB_{f_k} =\{x: |\det (Tf_k\mid_{E^{cu}_{f_k}(x)})|\le (\lambda^u_k)^2\}= I_k
$$
which has zero leaf volume inside any unstable leaf.
It follows from Theorem~\ref{m.boundedpotention} that for all $k$ large enough, there exists a $C^1$ open neighborhood $\cU_k$ of $f_k$, such that for every $g\in\cU_k$, the stable foliation of $g$ is minimal. In particular, $f_k$ is robustly transitive.

In this subsection we will prove the following lemma, which immediately leads to Proposition~\ref{proposition} (c):

\begin{lemma}\label{l.7.5}
Under the assumptions of Proposition~\ref{proposition} and the extra condition that $|c_k|/d_k<C$ for some $C>0$ and for all $k$ large enough, there exist $C^1$ partially hyperbolic diffeomorphisms $g_k$ arbitrarily close to $f_k$ in the $C^1$ topology, such that $g_k$ has robustly minimal unstable foliation.
\end{lemma}
 
We need the following definition:
\begin{definition}\label{d.usection}
	For a 3-dimensional partially hyperbolic diffeomorphism $g$, a {\it  $u$-section} of $g$ is  a compact surface $T$ with boundary such that $g(T)\subset \Int(T)$, and $\omega(T):=\cap_{n\geq0}g^n(T)$ is a finite union of center segments; if, in addition,   $\Int(T)$ has non-empty transversal intersection with every strong unstable leaf, then we say that $T$ is a  {\it complete $u$-section}.
\end{definition}

The forward invariance  implies that $T$ is tangent to $E^{cs}_{g}$; the set of diffeomorphisms having a complete $u$-section is $C^1$ open, see \cite[Proposition 3.1]{BDU}. 
 
We also enunciate the following  theorem of Bonatti, D\'iaz, and Ures.

\begin{theorem}\cite[Theorem 1.9]{BDU}\label{BDU} 
Let $M$ be a 3-dimensional compact Riemannian manifold and let $\mathcal V$ be an open subset of $\operatorname{Diff^1}(M)$ such that for every $g\in \mathcal V$:
\begin{itemize}
\item $g$ partially hyperbolic,
\item $g$ is transitive,
\item $g$  has a complete $u$-section.
\end{itemize}

Then there is a $C^1$ open and dense subset $\mathcal W\subset \mathcal V$ such that for every diffeomorphism in $\mathcal W$, the strong unstable foliation is minimal.

\end{theorem}

\begin{proof}[Proof of Lemma~\ref{l.7.5}]

We have already proved that $f_k$ is robustly transitive for large enough $k$.  Below we will show that under an arbitrarily small perturbation, a complete $u$-section can be created. 

Take a fundamental domain $D$ of $\TT^3$ in its universal covering. Recall from ~\eqref{base.k} that for $k$ sufficiently large the coordinate basis $\cB_k$ is close to the standard basis of $\RR^3$. Then we can assume that under the coordinate basis $\cB_k$, $D\subset [-\frac23,\frac23]^3\subset \mathbb{R}^3$. Let $\tilde \cF^\sigma_{f_k}$ be the lift of the foliation $\cF^\sigma_{f_k}$ to $\RR^3$, for $\sigma = s,u,c,cs,cu$.

We will apply Theorem \ref{BDU} to obtain the minimality of the strong unstable foliation. To that end, we will construct a complete $u$-section. In the next claim, we begin by constructing a non-compact local $u$-section in the universal cover. 

\medskip

\noindent{\bf Claim.} 
{\it For any $\zeta\in D$,
 $\tilde\cF^u_{f_k}(\zeta)\cap \bigcup_{x\in \Int(I_k)} \tilde\cF^s_{f_k}(x)\neq \emptyset$. }
 
 \medskip

 Let us assume for the moment that the claim is true. 
 Then the map $\psi_{f_{k}}$ that sends $\zeta\in  [-\frac23,\frac23]^3$ to the intersection between $\tilde\cF^u_{f_{k}}(\zeta)$ and $\bigcup_{x\in \Int(I_k)} \tilde\cF^s_{f_k}(x)$ is continuous with respect to $\zeta$ and with respect to $f_{k}$, when $f_{k}$ varies in the $C^{1}$-topology. $\psi_{f_{k}}$ is well defined because of the global product structure on $\RR^3$. As $f_k$ restricted to $I_k$  is the identity, there exists a perturbation $g_k$, arbitrarily close to $f_k$, such that $g_k(I_k)\subset\Int(I_k)$ and that $\omega(I_k) = \cap_{n\ge 0}g^n_k(I_k)$ is a subsegment of $I_k$. Then we denote by $T_{g_{k},R}$ the projection of $\tilde T_{g_{k},R}=\bigcup_{x\in I_k} \tilde\cF^s_{g_k,R}(x)$ to $\TT^3$; here for $R>0$, $\tilde \cF_{g_k,R}^s(x)$ denotes the $R$-ball in the leaf $\tilde \cF^s_{g_k}(x)$.  Since $D$ is compact and $\psi_{f_{k}}$ is continuous, $\psi_{f_{k}}(D)$ is compact, so for $R>0$ large enough $\tilde T_{f_{k},R}$ contains $\psi_{f_{k}}(D)$. The continuity of $\psi_{(.)}(.)$ implies that for sufficiently close $g_{k}$, $\tilde T_{g_{k},R}$ contains $\psi_{g_{k}(D)}$. Therefore, for all $x\in \TT^{3}$, its unstable leaf $\cF^{u}_{g_{k}}(x)$ intersects the projection $T_{g_{k},R}$ of $\tilde T_{g_{k},R}$. It is clear now that $T_{g_{k},R}$ is a complete $u$-section. The proof of Lemma~\ref{l.7.5} is then  complete. 
 
 It only remains to prove the claim.
 
\medskip
\noindent {\em Proof of the claim.} Firstly note that  by construction $f_k$  preserves both  $E^{cs}_{A_k} = \RR^{YZ} $ and $E^{cu}_{A_k}  = \RR^{XY}$. Secondly, recall   that under the coordinate basis $\cB_k$ we have $$I_k=\{0\}\times [-b_k,b_k]\times \{0\}.$$

From the previous subsection, the stable bundle is contained in the cone
$$\cC^s_{k}=\left\{\left[
    \begin{array}{c}
      0 \\
      u \\
      v \\
    \end{array}
  \right]_{\cB_k}: v\neq 0 \text{  and  }\frac{|u|}{|v|}\leq K^s_{k}\right\};$$ 
similarly, the unstable bundle is contained in the cone: 
 $$\cC^u_{k}=\left\{\left[
     \begin{array}{c}
     u \\
       v \\
      0 \\
     \end{array}
   \right]_{\cB_k}: u\neq 0 \text{  and  }\frac{|v|}{|u|}\leq K^u_{k}\right\}.$$ 
Since we assume, in addition, that $|c_k|/d_k<C$, ~\eqref{e.Ku} and~\eqref{e.Ks} become:
\begin{equation}\label{e.K}
\begin{split}
K^u_{k} &= \frac{\lambda^c_k}{\lambda^u_k}\cdot\frac{16|c_k|/d_k}{1-6\left(\frac{\lambda^c_k}{\lambda^u_k}\right)^2}\cdot b_k \le  \tilde C\frac{\lambda^c_k}{\lambda^u_k} b_{k},\\
K^s_{k} &= \lambda^s_k\lambda^c_k\cdot \frac{16|c_k|/d_k}{1-(\lambda^s_k)^2}\cdot b_k \le \tilde C \lambda^s_k\lambda^c_k b_{k};
\end{split}
\end{equation}
here $\tilde C>0$ is a constant independent of $k$. Since $\lambda^c_k/\lambda^u_k\searrow 0$ and $\lambda^s_k\lambda^c_k\searrow 0$ as $k\to\infty$,~\eqref{e.K} shows that
$$
\max\{K^s_{k},K^u_{k}\}= o(b_k).
$$

Since $\tilde\cF^s$ is tangent to the stable cone, for $k$ large we have
\begin{equation}\label{eq.cs}
\{0\}\times [-(1-o(1))b_k,(1-o(1))b_k]\times  [-1,1] \subset \bigcup_{x\in I_k} \tilde\cF^s_{f_k}(x).
\end{equation}

Now, for any $\zeta=[\zeta_x,\zeta_y,\zeta_z]_{\cB_k}\in D$, since the unstable leaf $\tilde \cF^u(\zeta)$ is tangent to the unstable $K^u_{k}$-cone, we have
\begin{equation}\label{e.11}
\tilde\cF^u(\zeta)\cap \mathbb{R}^{YZ}\in \{0\} \times [\zeta_y-o(b_k),\zeta_y+o(b_k)]\times \{\zeta_z\}.
\end{equation}
For $k$ sufficiently large, taking into account that  $|\zeta_y|\leq \frac{2}{3}$, and that $b_k>1$ (recall the choice of $\theta_k\in(0,(\lambda^c_k)^{-2})$ such that $b_k = \theta_ka_k>1$ at beginning of Section~\ref{ss.7.3}), we see that the set $\{0\} \times [\zeta_y-o(b_k),\zeta_y+o(b_k)]\times \{\zeta_z\}$  of~\eqref{e.11} is contained in the set $\{0\}\times [-(1-o(1))b_k,(1-o(1))b_k]\times  [-1,1]$ of~\eqref{eq.cs}, which implies that
\begin{equation}\label{eq.u}
\tilde\cF^u(y)\cap \mathbb{R}^{YZ} \in \bigcup_{x\in I_k} \tilde\cF^s_{f_k}(x).
\end{equation}
Now the proof of the  claim is complete, and Lemma~\ref{l.7.5} follows.
\end{proof}


\subsection{Center leaves of $A_k$}\label{ss.density}
In this subsection, we provide a density estimation of a long center leaf segment, which also justifies the possibility of having $|c_k|/d_k<C$.
Let us recall from Section~\ref{ss.7.3} that
$$
c_k = (\lambda^c_k)^{-2}-1\in(-1,0).
$$
Furthermore, $d_k>0$ is taken such that the projection from $\RR^3$ to $\TT^3$ is injective on the cylinder $\mathcal{T}_{a_k,d_k}$.  
Since 
$$
|c_k| = 1-\frac{1}{(\lambda^c_k)^2} = \frac{\lambda^c_k+1}{(\lambda^c_k)^2}(\lambda^c_k-1) = \mathcal{O}(\lambda^c_k-1),
$$
it suffices to prove the following lemma.

\begin{lemma}\label{l.separate} Let $v^{c}_{k}=\left(1,\frac{1}{\lambda^c_k-1}, \frac{1}{\lambda^c_k}\right)$.
For 
$$
a_k = \frac{1}{2}\left[\frac{1}{\lambda^c_k-1}\right]\cdot |v^c_k|\,\,\mbox{ and }\, d_k = \frac{\lambda^c_k-1}{4},
$$
the center segment
$
J_k = \left\{te^c_k: t\in [-a_k,a_k]\right\}
$, where $e^{c_{k}}=v^{c}_{k}/|v^{c}_{k}|$, 
satisfies 
$$
d(J_k,J_k+\mathbf n) >2d_k,
$$
for all $\mathbf n\in\ZZ^3\setminus \{0\}$. Therefore, $\pi:\RR^3\to\TT^3$ is injective on $\mathcal{T}_{a_k,d_k}$.  Moreover, it holds that
$$\label{eq.location}
\pi(J_k) \xrightarrow{k\to \infty} \TT^3 \mbox{ in Hausdorff topology}.
$$
\end{lemma}

\begin{proof}
From now on we will use the standard coordinate basis of $\RR^3$. To simplify notation, we define
$$
\tak = \frac{a_k}{|v^c_k|} = \frac12 \left[\frac{1}{\lambda^c_k-1}\right].
$$
Then
$$
J_k = \left\{ tv^c_k: t\in[-\tak,\tak]\right\}.
$$
We start with two observations. 
The first observation is that the distance between $J_k$ and $J_k + \mathbf n$ is invariant by translations and then, for the sake of simplicity, we will consider the segment
$$\ell_k:=\{t v^c_k: t\in[0,2\tak]\}$$ instead of $J_k$.  So, we will find a lower bound for $d(\ell_{k},\ell_{k}+\mathbf{n})=d(J_{k},J_{k}+\mathbf{n})$. The second observation is that, since we are taking integer translations,  we only need to compute the distances in $I^3=[0,1]\times[0,1]\times[0,1]$. Call $I^{2}=[0,1]\times [0,1].$
 
Call $P_Y$ the orthogonal  projection onto  the $xz$-plane; i.e.,
$$
P_Y(x,y,z) = (x,z).
$$ 
Then the projection of $\ell_k$ on $\RR^{XZ}$ is $$\ell^Y_k=
\left\{ P_Y(tv^c_k)= t \left(1,\frac{1}{\lambda^c_k}\right):t\in[0,2\tak] \right\}\subset\RR^{XZ}.$$

We claim that $\ell^Y_k \mod \ZZ^{2}\subset I^2\subset \RR^{XZ}$ only intersects the diagonal $\{(x,z)\in I^2:z=x\}$ at $t=0$. More precisely, $\ell^Y_k\setminus\{0,0\}$ lies strictly between the lines $z=x$ and $z=x-1$. This is because on the one hand $\lambda^{c}_{k}>1$, so $\frac{t}{\lambda^{c}_{k}}\ne t$ for $t\ne 0$. On the other hand, 
$$
\frac{1}{\lambda^c_k}t = t-1 \implies t= \frac{\lambda^c_k}{\lambda^c_k-1} > 2\tak.
$$
Therefore, $\ell^Y_k \mod \ZZ^{2}$ consists of two families of connected components: those above the diagonal, and those below it. Furthermore,  $\ell^Y_k$ hits the horizontal lattice $\{z=n\}$ and vertical lattice $\{x=n\}$ in an alternating fashion. See Figure \ref{projection}.  To describe these segments more precisely we  calculate the intersections of the integer translations of $\ell^Y_k$    with the boundary of the square $\partial I^2$.

\begin{figure}[h!]

\tikzset{every picture/.style={line width=0.75pt}} 

\begin{tikzpicture}[x=0.75pt,y=0.75pt,yscale=-1,xscale=1]

\draw  (48,252.8) -- (434,252.8)(86.6,8) -- (86.6,280) (427,247.8) -- (434,252.8) -- (427,257.8) (81.6,15) -- (86.6,8) -- (91.6,15)  ;
\draw (430, 264) node {$x$};
\draw (76, 25) node {$z$};

\draw    (86,65) -- (307,64) ;

\draw    (307,64) -- (307,252) ;

\draw    (86.6,252.8) -- (306,109) ;

\draw  [dash pattern={on 4.5pt off 4.5pt}]  (87,109) -- (306,109) ;

\draw    (87,109) -- (153,66) ;

\draw  [dash pattern={on 4.5pt off 4.5pt}]  (153,66) -- (153,251) ;

\draw    (153,251) -- (307,145) ;

\draw  [dash pattern={on 4.5pt off 4.5pt}]  (88,146) -- (307,145) ;

\draw    (88,146) -- (210,65) ;

\draw  [dash pattern={on 4.5pt off 4.5pt}]  (210,65) -- (211,251) ;

{\color{red}\draw  [dash pattern={on 2.5pt off 2.5pt}]  (307,65) -- (86,252) ;}

\draw    (211,251) -- (308,186) ;

\draw  [dash pattern={on 4.5pt off 4.5pt}]  (87,188) -- (308,186) ;

\draw    (87,188) -- (268,66) ;

\draw  [dash pattern={on 4.5pt off 4.5pt}]  (268,66) -- (267,252) ;

\draw    (306,225) -- (267,252) ;

\draw (52, 266) node {$x_0 = (0,0)$};
\draw (72,108) node {$\frac{1}{\lambda^{c}_{k}}$};
\draw (36, 147) node {$\frac{1}{\lambda^{c}_{k} }( 1-( \lambda^{c}_{k} -1))$};
\draw (33, 188) node {$\frac{1}{\lambda^{c}_{k} }( 1-2( \lambda^{c}_{k} -1))$};
\draw (155, 52) node {$x_1$};
\draw (210, 52) node {$x_2$};
\draw (267, 52) node {$x_3$};

\draw (324, 52) node {$(1,1)$};

\draw (324,108) node   {$z_0$};
\draw (120,127) node  [align=left] {};
\draw (149,269) node   {$\lambda^{c}_{k} -1$};
\draw (212,269) node   {$2( \lambda^{c}_{k} -1)$};
\draw (279,268) node   {$3( \lambda^{c}_{k} -1)$};
\draw (324,147) node   {$z_1$};
\draw (324,188) node   {$z_2$};
\draw (324,226) node   {$z_3$};

\end{tikzpicture}
 \caption{Orthogonal projection of $\ell^{Y}_k$ to the $xz$-plane.\label{projection}}

\end{figure}
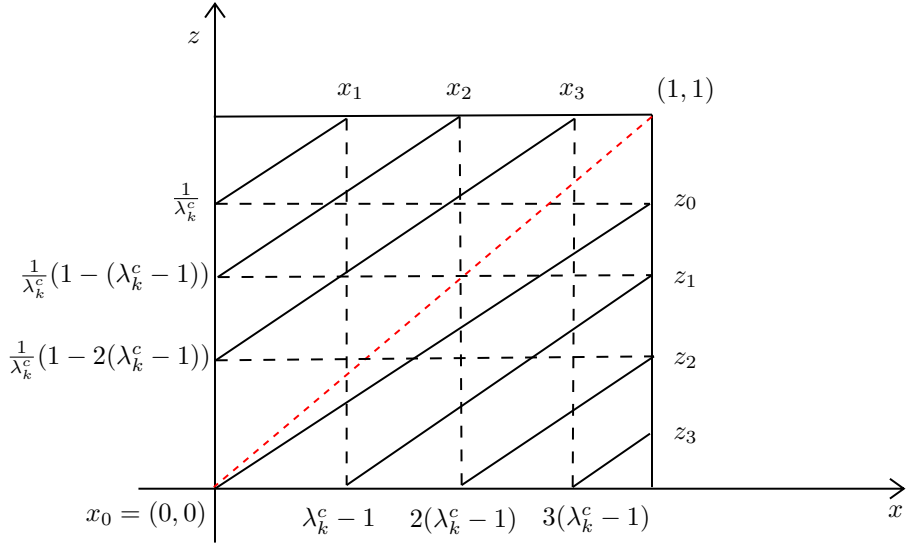

The first point of intersection of $\ell^{Y}_{k}$ with $\partial I^{2}$, called $x_0$, corresponds to $t=0$, that is 
$$x_0=(0,0)\sim (1,1) \sim (0,1)\sim (1,0),$$ where by $\sim$ we mean that the two points differ in an integer vector.  Since the slope of  $\ell^Y_k = 1/\lambda^c_k < 1$, we get the next point of intersection of $\ell^{Y}_{k}$ with $\partial I^{2}$, called $z_0$, at  $t=1$, that is 
$$z_0=\left(1,\frac1{\lambda^c_k}\right)\sim \left(0,\frac1{\lambda^c_k}\right).$$ The third point of intersection of $\ell^{Y}_{k} \mod \ZZ^{2}$ with $\partial I^{2}$, called  $x_1$, corresponds to $t=\frac1{\lambda^c_k}$, that is, 
$$x_1=(\lambda^c_k-1, 1)\sim (\lambda^c_k-1, 0).$$
At $t=2$ we obtain the forth point of intersection:
$$z_1=\left(1,\frac1{\lambda^c_k}(1-(\lambda^c_k-1))\right)\sim \left(0,\frac1{\lambda^c_k}(1-(\lambda^c_k-1))\right).$$ 
In this way we get two finite sequences of points:
$$x_n=(\,n(\lambda^c_k-1),0), \mbox{ and }$$ 
$$z_n=\left(1,\frac1{\lambda^c_k}(1-n(\lambda^c_k-1))\right),$$
with $0\leq n\leq n_k=2\tak -1$. 
Then all $x_n$'s are  $\lambda^c_k-1$ apart. 
Since for $k$ large enough the slope of $\ell^Y_k$ is very close to 1,  the distance between any two  segments both above or below the diagonal of $I^2$ is greater than $2d_k = (\lambda^c_k-1)/2$. Finally, the distance between the two segments closest to the diagonal is larger than the distance to a parallel line through $(1,1)$. Again for $k$ large enough this is close to $\frac{\sqrt 2}2(1-\frac1{\lambda^c_k})>2d_k$.

Since the distance of $x_{n_k}$ to $(1,0)\sim (1,1)$ is smaller than $2(\lambda^c_k-1)$, it is also easy to see that $\ell^Y_k \mod \ZZ^{2}$ is  $3(\lambda^c_k-1)$-dense in $I^2$.


Now we consider the three coordinates of the segment $\ell_k \mod \mathbb{Z}^{3}$, and represent it as a subset of the three dimensional cube $I^3=[0,1]^3$. 
We have
$$\ell_k \mod \mathbb{Z}^{3} \subset \{(x,y,z)\in I^3: (x,z)\in \ell^Y_k\} := S.$$
Then $S$ consists of rectangles, each of which projects to a connected component of $\ell^Y_k \mod \ZZ^{2}$ under $P_Y$.  
The distance between any two of these rectangles is at least $2d_k$, according to the previous argument. Furthermore, $S$ is $3(\lambda^c_k-1)$-dense in $I^3$.

Below we estimate the distance between two components of $\ell_k\mod \ZZ^{3}$ located in the same connected component of $S$. Again, considering the invariance of the distance under translations, it is enough to focus on  the rectangle that contains the origin: 
$$S_0= \left\{\left(t, s, \frac{t}{\lambda^c_k}\right): \,s,t\in [0,1]\right\}.$$

In order to estimate the distance between two components of $\ell_k\mod \ZZ^{3}$ inside $S_0$, we have to compute at what points $\ell_k\mod \ZZ^{3}$  hits the right side of $S_0$, which is the side contained in the plane $y=1$ and is identified with the left side of $S_0$, which is the side contained in the plane $y=0$. \newline\par
Let
$$L:=\{(t,1,t/\lambda^c_k):t\in[0,1]\}\sim \{(t,0,t/\lambda^c_k):t\in[0,1]\}.$$
In Figure \ref{s_0}, we draw $\ell_k\mod \ZZ^{3}$ inside the rectangle $S_0$. The vertical axis in the figure is the line $L$. The horizontal axis in the figure is the $y$-axis. 


\begin{figure}[h]
\tikzset{every picture/.style={line width=0.75pt}} 

\begin{tikzpicture}[x=0.75pt,y=0.75pt,yscale=-1,xscale=1]\label{s_0}

\draw    (200,350) -- (538,349.01) ;
\draw [shift={(540,349)}, rotate = 539.8299999999999] [color={rgb, 255:red, 0; green, 0; blue, 0 }  ][line width=0.75]    (10.93,-3.29) .. controls (6.95,-1.4) and (3.31,-0.3) .. (0,0) .. controls (3.31,0.3) and (6.95,1.4) .. (10.93,3.29)   ;

\draw    (199,89) -- (200,350) ;

\draw    (199,89) -- (459,90) ;

\draw    (459,90) -- (460,349) ;

\draw    (460,290) -- (200,350) ;

\draw  [dash pattern={on 4.5pt off 4.5pt}]  (200,291) -- (460,290) ;

\draw    (460,231) -- (200,291) ;

\draw  [dash pattern={on 4.5pt off 4.5pt}]  (460,231) -- (200,230) ;

\draw    (459,170) -- (200,230) ;

\draw  [dash pattern={on 4.5pt off 4.5pt}]  (200,171) -- (459,170) ;

\draw    (200,171) -- (460,112) ;

\draw  [dash pattern={on 4.5pt off 4.5pt}]  (199,111) -- (460,112) ;

\draw    (199,111) -- (282,90) ;

\draw (460,364) node   {$(0,1,0)$};
\draw (198,364) node   {$(0,0,0)$};
\draw (480,292) node   {$w_1$};
\draw (480,232) node   {$w_2$};
\draw (480,172) node   {$w_3$};
\draw (480,113) node   {$w_4$};
\draw (270, 343)  node {$\beta_k$};
\draw (165, 80) node {$\left(1,0,\frac{1}{\lambda^c_k}\right)$};
\draw (330, 70) node {$y_0 = \left(1,\frac{1}{\lambda^c_k-1},\frac{1}{\lambda^c_k}\right)$};

\end{tikzpicture}
\caption{$\ell_k\mod \ZZ^{3}$ inside $S_0$\label{s_0}}
\end{figure}

Recall that $v^c_k=(1, 1/(\lambda^c_k-1), 1/{\lambda^c_k})$. Then $\ell_k\mod \ZZ^{3}$ first hits $L$ when $t_1=\lambda^c_k-1$. The corresponding point is 
$$
w_1 = \left(\lambda^c_k-1,1, \frac{\lambda^c_k-1}{\lambda^c_k}\right).
$$
Similarly, we have a sequence of points of intersection between $\ell_k\mod \ZZ^3$ and $L$
$$
w_n = \left(n(\lambda^c_k-1),1,n\frac{\lambda^c_k-1}{\lambda^c_k}\right),\qquad n=1,\ldots, \left[\frac{1}{\lambda^c_k-1}\right],
$$
given by $t_n = n(\lambda^c_k-1)$.
After that, $\ell_k\mod \ZZ^{3}$ hits the top side of $S_0$ at $y_0 = (1,1/(\lambda^c_k-1),1/\lambda^c_k)$ which is given by $t=1$, and then it jumps to another connected component of $S$.

The distance between two consecutive $w_n$ is the same for all $n$, and it is the same as
\begin{equation}\label{e.dz}
d(w_1,w_2)= \sqrt{(\lambda^c_k-1)^2+ \left(\frac{\lambda^c_k-1}{\lambda^c_k}\right)^2} = (\lambda^c_k-1)\sqrt{1+(\lambda^c_k)^{-2}}.
\end{equation}
Denote by 
$$\beta_k = \arctan \left(d(w_1,w_2)\right)
$$
the angle between components of  $\ell_k\mod \ZZ^{3}$ and the top/bottom side of $S_0$. By~\eqref{e.dz} we have $\beta_k\to 0$ as $k\to\infty$. 
Call $m_{k}$ the minimum distance between two components of $\ell_k\mod \ZZ^{3}$ inside $S_0$.  
Then for sufficiently large $k$,
$$m_{k}= \sin \beta_{k}\ge \frac12 \tan \beta_k= \frac12(\lambda^c_k-1)\sqrt{1+(\lambda^c_k)^{-2}}> \frac12(\lambda^c_k-1) = 2d_k.$$
With this we conclude the proof that $d(\ell_k,\ell_k+\mathbf n)>2d_k.$ 

Note that from~\eqref{e.dz}, we have $d(w_1,w_2) < 2(\lambda^c_k-1)$. It follows that $\ell_k\mod \ZZ^{3}$ is $2(\lambda^c_k-1)$-dense in $S_0$, and by translation invariance, in all connected components of $S$. Since $S$ itself is $3(\lambda^c_k-1)$-dense in $I^3$ by the first part of the proof, it follows that $\ell_k\mod\ZZ^{3}$ is $5(\lambda^c_k-1)$-dense in $I^3$. This shows that $\lim_H \ell_k = I^3$, concluding the proof of Lemma~\ref{l.separate}.

\end{proof}

Now the proof of Proposition~\ref{proposition} is complete, and Theorem~\ref{main.example} follows.

\appendix

\section{Proof of Lemma~\ref{l.dominated}}
We recall several well-known
facts about partially hyperbolic diffeomorphisms. First, we provide an alternate  definition of partial hyperbolicity as oppose to \eqref{pointwise.ph}.
The equivalence of these two definitions, changing the metric if necessary, was shown by Gourmelon \cite{G}. We say a diffeomorphism $f\in \Diff(\TT^3)$ admits a 
\emph{dominated splitting}  $T\TT^3=E^s\oplus E^c\oplus E^u$ if there are $C>0$ and  $\lambda_2<\lambda_3\leq\lambda_4<\lambda_5$
such that for any $x\in \TT^3$ and any $n>0$:
 \begin{equation}\label{eq.definitionofabphd}
\begin{split}
&\|Tf^n\mid_{E^s(x)}\|\leq C e^{\lambda_2n},\\
C^{-1
} e^{\lambda_3n}\leq & \|Tf^n\mid_{E^c(x)}\|\leq C e^{\lambda_4n},\\
C e^{\lambda_5n}\leq &\|Tf^n\mid_{E^u(x)}\|.
\end{split}
\end{equation}
Furthermore, $f$ is \emph{partially hyperbolic} if $\lambda_2<0$ and $\lambda_5>0$.

According to our hypotheses, we have two invariant two-dimensional subbundles $E$ and $G$ which are transverse to each other
at any point. We are also assuming there are closed cones $\cC^E\subset E$ and $\cC^G\subset G$ that are both transverse to the one-dimensional bundle $F=E\cap G$.
We want to establish that $f$ admits a dominated splitting $T\TT^3=E^\prime\oplus F \oplus G^\prime$, where $E^\prime=\cap_{n\ge 0} Tf^{-n}(\cC^E)$
and $G^\prime=\cap_{n\ge 0} Tf^n(\cC^G)$.  We will use the equivalent definition for a dominated splitting stated in ~\eqref{eq.definitionofabphd}.

Moreover, we want to prove that $f$ is partially hyperbolic if for any $x\in \TT^3$ we have
\begin{equation*}
\begin{split}
&|\det Tf\mid_{E(x)}|<|\det Tf\mid_{F(x)}|=\|Tf\mid_{F(x)}\|\\ &\text{  and  } \\ &|\det Tf\mid_{G(x)}|>|\det Tf\mid_{F(x)}|=\|Tf\mid_{F(x)}\|.
\end{split}
\end{equation*}

By  \cite[Theorem B]{BG} (see also \cite[Theorem 2.6]{CP}) we see that $Tf\mid_G$ has a dominated splitting $F\oplus G'$. Therefore, there exists $C>0$, $\lambda_4<\lambda_5$ with
\begin{equation}\label{e.a2}
\|Tf^n\mid _{F}\|\le Ce^{\lambda_4n}, \mbox{ and } Ce^{\lambda_5n}\le \|Tf^n\mid _{G'}\|.
\end{equation}
Similarly, applying the same argument on $E$, we obtain that $Tf|_{E}$ admits a dominated splitting $E'\oplus F$, that is, for some $C'>0$ and $\lambda_2<\lambda_3$,
\begin{equation}\label{e.a3}
\|Tf^n\mid _{E'}\|\le C'e^{\lambda_2n}, \mbox{ and } C'e^{\lambda_3n}\le \|Tf^n\mid _{F}\|.
\end{equation}
Combining~\eqref{e.a2} and~\eqref{e.a3} and noting that $\lambda_3\le\lambda_4$ otherwise $F$ would be empty,  we conclude that $E'\oplus F\oplus G'$ is a dominated splitting.

It remains then to show that $T\TT^3=E^\prime\oplus F \oplus G^\prime$ is a partially hyperbolic splitting.
We need to prove that it is possible to take $\lambda_2<0$ and $\lambda_5>0$. We will only show the first
 one; the other case is similar.
 
Suppose by contradiction that we cannot take $\lambda_2<0$ in~\eqref{e.a3}. This implies that
there exists $x\in \TT^3$ such that for any $n\geq 0$, $\log\|T f^n\mid_{E^\prime(x)}\|\geq 0$. Otherwise,  for each point $y$ there is an iterate $n_y>0$ for which $\log \|Tf^{n_y}|_{E^\prime(y)}\|<0$.  The continuity of $\log \|Tf^{n}|_{E^\prime(.)}\|$ for each $n\in\NN$, and compactness of $\TT^3$, implies that we can choose the $n_y$ such that they are bounded. This implies that the bundle $E^\prime$ is hyperbolic (contracting), contradicting our assumption. 

Let $x$ be the point obtained above. Take  $\mu$  to be 
any  weak$^*$ limit  of the sequence of probability measures $\{\frac{1}{n}\sum_{i=0}^{n-1}\delta_{f^i(x)}\}$. Note that $\mu$ is $f$-invariant.
The function  $\log \|Tf|_{E^\prime(.)}\|$ is continuous, so
\begin{equation}\label{eq.nonhyperbolic}
\int \log ||Tf\mid_{E^\prime(x)}||\,d\mu(x)\geq 0.
\end{equation}

   Oseledec's Theorem implies that  the splitting $E^\prime\oplus F\oplus G^\prime$
coincides with  Oseledec's splitting for the measure $\mu$. 
Therefore, at  $\mu$-almost every point $y$, there are well defined Lyapunov exponents  
 $\kappa_1(y)<\kappa_2(y)<\kappa_3(y)$ corresponding to the bundles $E',F$ and $G'$, respectively. 
 These exponents satisfy 
\[\begin{split} 
 K_1&=\int \kappa_1(y) \,d\mu(y)=\int \log ||Tf\mid_{E^\prime(y)}||\,d\mu(y), \\
  K_2&=\int \kappa_2(y) \,d\mu(y)=\int \log ||Tf\mid_{F(y)}||\,d\mu(y), \mbox{ and}\\
  K_1+K_2&= \int \log |\det Tf\mid_{E^\prime(y)\oplus F(y)}|\,d\mu(y)=\int \log |\det Tf\mid_{E(y)}|\,d\mu(y).
 \end{split}
\]
 
By \eqref{eq.partialhyperbolic}, $K_1+K_2<K_2$, which implies $K_1<0$.
But by \eqref{eq.nonhyperbolic}, $K_1\geq 0$, a contradiction.

\section*{Acknowledgments} The authors thank Fan Yang for many helpful suggestions, especially regarding the example in Section 7. We also thank  the anonymous referees for their careful reading and helpful comments, which significantly improved the presentation of the current paper.


\end{document}